\documentclass[12pt]{iopart}
\usepackage{iopams}
\usepackage{setstack}
\usepackage{amsfonts}
\usepackage{amsthm}

\usepackage{amssymb}
\usepackage{graphicx}
\usepackage{epsfig}
\usepackage{psfrag}
\usepackage{hyperref}

\newtheorem{theorem}{Theorem}

\newtheorem{definition}[theorem]{Definition}
\newtheorem{lemma}[theorem]{Lemma}

\theoremstyle{remark}
\newtheorem{remark}[theorem]{Remark}

\def\be{\begin{equation}}
\def\ee{\end{equation}}
\def\bea{\begin{eqnarray}}
\def\eea{\end{eqnarray}}

\def\hsm1{\hspace{-1mm}}

\makeatletter
\providecommand\underarrow@[3]{%
  \vtop{\ialign{##\crcr$\m@th\hfil#2#3\hfil$\crcr
  \noalign{\nointerlineskip\kern.12\baselineskip}#1#2\crcr}}}
\providecommand{\underrightarrow}{%
  \mathpalette{\underarrow@\rightarrowfill@}}
\providecommand\rightarrowfill@{\arrowfill@\relbar\relbar\rightarrow}
\providecommand\arrowfill@[4]{%
  $\m@th\thickmuskip0mu\medmuskip\thickmuskip\thinmuskip\thickmuskip
   \relax#4#1\mkern-7mu%
   \cleaders\hbox{$#4\mkern-2mu#2\mkern-2mu$}\hfill
   \mkern-7mu#3$%
}
\makeatother

\begin{document}

\title[Arnold diffusion in the planar elliptic restricted three-body problem]{Arnold diffusion in the planar elliptic restricted three-body problem: mechanism and numerical verification}
	
\author{Maciej J. Capi\'nski}

\address{Faculty of Applied Mathematics, AGH University of Science and Technology, Mickiewicza 30, 30-059 Krak\'{o}w, Poland}
\ead{maciej.capinski@agh.edu.pl}

\author{Marian Gidea}

\address{School of Civil Engineering and Architecture, Xiamen University of Technology, Fujian, China, and\\
Department of Mathematical Sciences, Yeshiva University, New York, NY 10018, USA}
\ead{marian.gidea@yu.edu}

\author{Rafael de la Llave}

\address{School of Mathematics, Georgia Institute of Technology, 686 Cherry St NW, Atlanta, GA 30332, United States}

\ead{rll6@math.gatech.edu}

\begin{abstract}
We present a diffusion mechanism for time-dependent perturbations of autonomous Hamiltonian systems introduced in \cite{GideaLlaveSeara14}.  This mechanism is based on shadowing of pseudo-orbits generated by two dynamics: an `outer dynamics', given by homoclinic trajectories to a normally hyperbolic invariant manifold, and an `inner dynamics', given by the restriction to that manifold. On the inner dynamics the only assumption is that it preserves area. Unlike other approaches, \cite{GideaLlaveSeara14} does not rely on the KAM theory and/or Aubry-Mather theory to establish the existence of diffusion. Moreover, it does not require to check twist conditions or non-degeneracy conditions near resonances. The conditions are explicit and can be checked by finite precision calculations in concrete systems (roughly, they amount to checking that Melnikov-type integrals do not vanish and that some manifolds are transversal).

As an application, we study the planar elliptic restricted three-body problem. We present a rigorous theorem that shows that if some concrete calculations yield a non zero value, then  for any  sufficiently small,  positive value of the eccentricity of the orbits of the main bodies, there are orbits of the infinitesimal body that exhibit a change of energy that is bigger than some fixed number, which is independent of the eccentricity.

We verify numerically these calculations for values of the masses close to that of the Jupiter/Sun system. The numerical calculations  are not completely rigorous, because we ignore issues of round-off error and do not estimate the truncations, but they are not delicate at all by the standard of numerical analysis. (Standard tests indicate that we get 7 or 8 figures of accuracy where 1 would be enough). The code of these verifications is  available. We hope that some full computer assisted proofs will be obtained in a near future since there are packages (CAPD) designed for problems of this type.
\end{abstract}


\section{Introduction\label{sec:introduction}}

The goal of this paper is to show existence of Arnold diffusion in a problem
in celestial mechanics, namely the planar restricted three body problem. When the
 two main  bodies move in  circular orbits, 
the energy of the infinitesimal body in the rotating coordinates does not change. We want to show that when the main
bodies move in elliptic orbits (no matter how small is the eccentricity) we
can obtain changes of energy of the infinitesimal  body of order $1$. These effects
on comets have been considered in several papers \cite{GalanteK1,Sun,GalanteU}. 
In this paper we want to consider the changes of energies in the
family of Lyapunov periodic orbits around the equilibrium point $L_2$, which present interest in astrodynamics, 
and show that if the eccentricity of the orbits of the main bodies is not zero then there are
orbits of the infinitesimal body that for negative time accumulate in one Lyapunov orbit and for
positive time accumulate in another. The difference in energy of the initial
and final orbit is independent of the eccentricity of the orbits of the main bodies. Of course, we  also
obtain  that the energy can evolve in rather arbitrary ways.

In Section~\ref{sec:diffusion-mechanism} we review the general mechanism in
\cite{GideaLlaveSeara14}, that shows the existence of these orbits provided
that some finite number of very concrete conditions are verified.

In Section~\ref{sec:3bp} we study the planar elliptic restricted three body
problem (henceforth PER3BP) and work out the explicit form of the conditions
in \cite{GideaLlaveSeara14}. We take advantage of the fact that for the
planar circular restricted three body problem, the orbits can be computed
rather explicitly and we also take advantage of the fact that the PCR3BP is
reversible.  We formulate a theorem (Theorem \ref{th:diffusion-3bp}) which
ensures that the diffusion will take place under appropriate assumptions.

In Section~\ref{sec:num} we discuss the numerical verification of the
assumptions of Theorem~\ref{th:diffusion-3bp}.  We anticipate that the
conditions amount to check that certain manifolds intersect transversally, and to the calculation of several explicit integrals along trajectories of the unperturbed problem and checking that they are not zero.

These calculations are, by today's standards, rather straightforward,  can be
performed quite comfortably, and are very well into the safety region. We
have not performed fully rigorous computer assisted proofs, but standard
numerical analysis checks indicate that our calculations have errors
not bigger that $10^{-6}$ -- $10^{-8}$ times the significant value which we want to show it is not zero. 
We make the code available.

Obtaining a full computer assisted proof does not seem difficult since there
are libraries such as CAPD\footnote{%
http://capd.ii.uj.edu.pl/}, which are designed to do this,  and  indeed more
complicated problems have been dealt with computer assisted proofs \cite%
{Ca1, Ca3, Ca2,  Fig,Jay, W3, W2, WZ, W1}. We discuss the possibilities of a
computer assisted proof in Section~\ref{sec:future-work}.

\subsection{Other recent papers on the PER3BP}

The problem of the behavior of the PER3BP has a 300 years history since it
indeed appears prominently in \cite{Laplace1, Laplace2, Principa} and we
cannot attempt to do a survey (See \cite{Marchal} for a partial one). Also
the field of Arnold diffusion has experienced significant growth in the last
20 years.

Here, we present only some small remarks on papers which are closer to the
present paper.

A related result is presented in \cite{CZ}. There, the PER3BP is also
considered. The method relies upon a computation of a Melnikov integral
along a homoclinic orbit to one of the libration points of the PCR3BP. Such
homoclinic orbit exists only for a selection of specially chosen mass
parameters \cite{Simo}. In \cite{CZ} it is shown that we have diffusion for
those parameters, for with the mass of one of the two larger bodies is
sufficiently small. The diffusion shown in \cite{CZ} can only overcome small
ranges of energies, the size of which reduce to zero together with the size
of the perturbation. This is often referred to as `micro diffusion'. The
current paper is a major improvement with respect to \cite{CZ}. Firstly, we
do not need the homoclinic orbit to the libration point, and can thus work
with a broad range of mass parameters. For instance, in this paper we
consider the Jupiter-Sun system, but other systems, say, the Earth-Moon
system, also possess the same desirable properties and our mechanism can also
be used there. Secondly, we do not need to assume that the mass of one of
the two larger masses is sufficiently small. Lastly, and most importantly,
the mechanism from this paper ensures diffusion along a given range of
energies, which is independent of the size of the perturbation.

A different mechanism of diffusion is described in \cite{DGR15}, in the case
of the spatial restricted three-body problem (see also \cite{DGR10}). The paper  \cite{DGR15} also
focuses on one of the equilibrium points near which there exists a NHIM,
which is diffeomorphic to a three-dimensional sphere when restricted to an
energy level. The three-dimensional sphere can be parametrized by one action
and two angle coordinates. The paper \cite{DGR15} shows numerically that the
stable and unstable manifolds of the NHIM intersect transversally, and that
the dynamics restricted to the NHIM satisfies a twist condition. An
important observation is that the twist in this example is very weak. Then,
combining the scattering map associated to a choice of a transverse
homoclinic intersection and the twist map, it is shown numerically that
there exist trajectories whose projection onto the three-sphere change their
action coordinate by some `non-trivial' amount. Physically, these
trajectories change their out-of-plane amplitude of motion from close to
zero to close to the maximum possible out-of-plane amplitude for that
energy. There are significant differences between the diffusion mechanism of
this paper and that of \cite{DGR15}. First of all, we deal with a different
model, the PER3BP, which is given by a time-periodic Hamiltonian. Second,
for the presented mechanism it is not essential that the inner dynamics is
given by a twist map. Third, we use several scattering maps rather than a
single one. This allows us to minimize the time that we use the inner
dynamics, which is slower, and follow mostly the outer dynamics, which is
faster.

A recent important development in the search for diffusion in the restricted
three body problem is given in \cite{Kaloshin}, where a different setting is
considered. Instead of the Lyapunov orbits, the outer periodic orbits of the
system are considered. (Such orbits in real life can correspond to the
asteroid belt which is located between the orbits of Mars and Jupiter.) The
diffusion mechanism follows from careful analysis of the properties of the
circular problem, together with perturbation results, some of which (as is
the case in this paper) are only verified numerically. We have heard that 
the authors of \cite{Kaloshin} are also considering validating their results
to produce a computer assisted proof. One of the difficulties in such proof
could be the fact that the considered manifolds intersect at a very small
angle. A computer assisted validation of such a splitting could constitute a
numerical difficulty. In our approach we are free of this problem, since the
angle of intersection of manifolds considered by us is very large. (In
fact, the proof of transversality and the rigorous-computer-assisted
validation of such intersections have already been done in \cite{Ca1}.)

One can also mention the study of behavior of comets in the PRE3BP, which
are different orbits from the ones considered here. Early delicate numerical
studies were done by \cite{Sun}. Several remarkable studies were done in
\cite{GalanteK1, GalanteU}. Most of these studies are numerical, but \cite%
{GalanteK1} succeeded in presenting computer assisted proofs of most of the
steps needed to verify the assumptions of Mather's instability theorem.
Other related works include \cite{delaRosa,Xue}.

\subsection{A general mechanism for diffusion}

We consider a Hamiltonian system which possesses a normally hyperbolic
invariant manifold foliated by invariant tori, and whose stable and unstable
manifolds intersect transversally. We present a diffusion mechanism for time
dependant perturbations of such a system, showing that, for all sufficiently
small perturbations, there exist trajectories whose energy changes between
some given levels that are independent of the size of the perturbation.

The main ingredient of the mechanism is the scattering map. Given a
transverse intersection of the stable and unstable manifolds which also
satisfy an extra transversality condition, one can associate to it a mapping
from a subset of the NHIM to the NHIM called the scattering map. Given a
point in the intersection of stable and unstable manifolds, the scattering
map (quite analogous to the scattering matrix in Quantum mechanics)
associates to the asymptotic orbit in the past, the asymptotic orbit in the
future.

The scattering map enjoys remarkable geometric properties and, in the
perturbative setting can be computed very efficiently using Melnikov
integrals \cite{DLS1}.

The main assumption of \cite{GideaLlaveSeara14} is that we can find a finite
collection of scattering maps so that each map can move by a significant distance by 
applying   successive iterates. 
This is easy to verify in the
perturbative setting because the scattering maps are close to identity, so
that the problem is very similar to that of accessibility under several
controls \cite{Sussmann}. The main result of \cite{GideaLlaveSeara14} is
that, if the dynamics in the NHIM satisfies Poincar\'{e} recurrence, then,
one can find orbits that shadow the orbit of scattering maps.

This is enough to establish diffusion, because if the Poincar\'e recurrence
fails or if the perturbed manifold is only locally invariant, then, we
obtain diffusion by orbits on the locally invariant manifold.

Hence, verifying just the properties of the Melnikov integrals and the
transversality gives diffusion either by the jumping mechanism or by
diffusion on the normally hyperbolic manifold.

This method could be compared with other methods based either on geometric
considerations or on variational considerations. The most striking
difference is that the present method does not rely on sophisticated tools
such as Aubry-Mather and KAM theory. Hence, it is not needed to verify the
twist conditions nor other non-degeneracy conditions on the resonances.

As a matter of fact, for the problem at hand, we can indeed verify the twist
conditions and the hypothesis of the KAM theorem (at least for a range of
energies). Hence we conclude that for these ranges of energies the diffusion
along the stable manifold is impossible and we have diffusion using the
jumping mechanism.


\section{Preliminaries}

\label{sec:prel} In this section we gather several standard results that are
used in the paper and set the notation. The techniques that are used in this
paper are mainly normally hyperbolic manifolds and the scattering map. KAM
plays a very minor role. Of course, this section can be omitted in a first
reading and can be used mainly as reference.

Let $f=(f_1,\ldots,f_m):M\rightarrow \mathbb{R}^{m}$ be a $C^{r}$ function,
where $M$ is an $n$-dimensional manifold (not necessarily compact). We write
\begin{equation*}
\Vert f\Vert _{C^{r}}=\sup_{p\in M}\max \left\{ \left\Vert \frac{\partial
^{|k|}f_{i}}{\partial x_{1}^{k_{1}}\ldots \partial x_{n}^{k_{n}}}%
(p)\right\Vert :i\in \{1,\ldots ,m\},|k|\leq r\right\} ,
\end{equation*}%
where $k=(k_{1},\ldots ,k_{n})$ is a multi-index and $|k|=k_{1}+\ldots
+k_{n} $. The $\Vert \cdot \Vert _{C^{r}}$ norm induces a topology on the
space of $C^{r}$ functions, which we refer to as the uniform $C^{r}$%
-topology. Note that for this space we not only require that the derivatives
are continuous but also that they are uniformly bounded.

When $M$ is compact, as it will be in our setting, the uniform $C^{r}$-topology is the same as the $C^{r}$%
-topology.

\subsection{Normally hyperbolic invariant manifolds}

The following results and definitions are standard. They were introduced in
\cite{Fenichel, HPS}. A more recent tutorial survey is \cite{nhims}. Modern
proofs which also lead to constructive algorithms appear in \cite%
{Capinski2009,CapinkiZ2011,CapinskiZ2015}.

\begin{definition}
\label{def:nhim}Let $\Lambda \subset \mathbb{R}^{n}$ be a manifold,
invariant under $f:\mathbb{R}^{n}\rightarrow \mathbb{R}^{n}$, i.e., $%
f(\Lambda )=\Lambda $, where $f$ is a $C^{r}$-diffeomorphism, $r\geq 1$. We
say that $\Lambda $ is a normally hyperbolic invariant manifold (with
symmetric rates) if there exists a constant $C>0,$ rates $0<\lambda <\mu
^{-1}<1$ and a splitting for every $\mathbf{x}\in \Lambda $%
\begin{equation*}
\mathbb{R}^{n}=E_{\mathbf{x}}^{u}\oplus E_{\mathbf{x}}^{s}\oplus T_{\mathbf{x%
}}\Lambda
\end{equation*}%
such that%
\begin{eqnarray}
v &\in &E_{\mathbf{x}}^{u}\Leftrightarrow \left\Vert Df^{n}(\mathbf{x}%
)v\right\Vert \leq C\lambda ^{|n|}\left\Vert v\right\Vert ,\quad n\leq 0,
\label{eq:rate-cond-nhim1} \\
v &\in &E_{\mathbf{x}}^{s}\Leftrightarrow \left\Vert Df^{n}(\mathbf{x}%
)v\right\Vert \leq C\lambda ^{n}\left\Vert v\right\Vert ,\quad n\geq 0,
\label{eq:rate-cond-nhim2} \\
v &\in &T_{\mathbf{x}}\Lambda \Leftrightarrow \left\Vert Df^{n}(\mathbf{x}%
)v\right\Vert \leq C\mu ^{|n|}\left\Vert v\right\Vert ,\quad n\in \mathbb{Z}.
\label{eq:rate-cond-nhim3}
\end{eqnarray}
\end{definition}

We note that the definition of normally hyperbolic invariant manifold with
symmetric rates is less general than that of \cite{Fenichel2,Fenichel4}.
Nevertheless it is very natural for symplectic systems.

In the sequel we will assume that $\Lambda $ is compact.
Let $U$ be a sufficiently small neighborhood of $\Lambda$. 

Given a normally hyperbolic invariant manifold we define its unstable and
stable manifold as%
\begin{eqnarray*}
W^{u}\left( \Lambda ,f\right) &=&\left\{ \mathbf{y}\in U\,|\,d\left( f^{k}(\mathbf{y}),\Lambda \right) \leq C_{\mathbf{y}}\lambda
^{\left\vert k\right\vert },k\leq 0\right\} , \\
W^{s}\left( \Lambda ,f\right) &=&\left\{ \mathbf{y}\in U\,|\,d\left( f^{k}(\mathbf{y}),\Lambda \right) \leq C_{\mathbf{y}}\lambda
^{k},k\geq 0\right\} .
\end{eqnarray*}%
The manifolds $W^{u}\left( \Lambda ,f\right) ,$ $W^{s}\left( \Lambda
,f\right) $ are foliated by%
\begin{eqnarray*}
W^{u}\left( \mathbf{x},f\right) &=&\left\{ \mathbf{y}\in U\,|\,d( f^{k}(\mathbf{y}),f^{k}(\mathbf{x})) \leq C_{%
\mathbf{x},\mathbf{y}}\lambda ^{\left\vert k\right\vert },k\leq 0\right\} ,
\\
W^{s}\left( \mathbf{x},f\right) &=&\left\{ \mathbf{y}\in U\,|\,
d( f^{k}(\mathbf{y}),f^{k}(\mathbf{x})) \leq C_{%
\mathbf{x},\mathbf{y}}\lambda ^{k},k\geq 0\right\} .
\end{eqnarray*}

Let
\begin{equation}
l<\min \left\{ r,\frac{\beta }{\alpha }\right\} .
\label{eq:nhim-smoothness}
\end{equation}

The manifold $\Lambda $ is $C^{l}$ smooth, the manifolds $W^{u}\left(
\Lambda ,f\right) ,W^{s}\left( \Lambda ,f\right) $ are $C^{l-1}$ and $%
W^{u}\left( \mathbf{x},f\right) $, $W^{s}\left( \mathbf{x},f\right) $ are $%
C^{r}.$

Assume that there exist a $C^{r-1}$ diffeomorphism onto its image $k_{0}:%
\mathcal{N}\rightarrow \mathbb{R}^{n}$ with $k_{0}(\mathcal{N})=\Lambda$,
and a $C^{r-1}$ diffeomorphism $r_{0}:\mathcal{N}\rightarrow \mathcal{N}$
such that%
\begin{equation*}
f\circ k_{0}=k_{0}\circ r_{0}.
\end{equation*}%
Then we refer to $\mathcal{N}$ the reference manifold for $\Lambda$, to $%
k_{0}$ the parametrization of $\Lambda $, and to $r_{0}$ as the inner
dynamics induced by $f_{\mid \Lambda}$ on $N$. The following theorem gives a
persistence and smoothness result for perturbations of a map with a normally
hyperbolic invariant manifold.

\subsubsection{Persistence of normally hyperbolic manifolds}

The main property of invariant manifolds is that they persist under
perturbations. Even if the definitions presented so far work the same when $%
\Lambda$ has a boundary and when it does not, for the theory of persistence
there is a difference between the two cases. We will first state the theory
for the case of boundaryless manifolds and then describe the modifications
needed to allow for boundaries. The case of manifolds with boundary is
treated explicitly in \cite{Fenichel} (see also \cite{nhims} for a more
recent survey). In our case the manifolds do have boundary.

The following result is the main theorem of persistence of manifolds without
boundary. We have taken the statement from \cite[Theorem 23]{DLS1}, which
gathers it from different sources referenced there.

\begin{theorem}[Normally hyperbolic invariant manifold theorem]
\label{th:nhim-pert} Let $f_{\varepsilon}:\mathbb{R}^{n}\rightarrow \mathbb{R%
}^{n}$ be a family of $C^{r}$ diffeomorphisms with $r\geq 2.$ Assume that $%
\Lambda$ is a normally hyperbolic invariant manifold for $f_{0}$ with rates $%
\lambda,\mu$, a reference manifold $\mathcal{N},$ parametrization $k_{0}$
and inner dynamics $r_{0}.$ Then there exists an $\varepsilon_{0}>0$, such
that for $f_{\varepsilon}$ which are $\varepsilon_{0}$ close to $f_0$ in the
$C^r$-topology, there exist $C^{l-1}$ families $k_{\varepsilon}:\mathcal{N}%
\rightarrow\mathbb{R}^{n},r_{\varepsilon}:\mathcal{N}\rightarrow\mathcal{N}$
satisfying%
\begin{equation*}
f_{\varepsilon}\circ k_{\varepsilon}=k_{\varepsilon}\circ r_{\varepsilon}.
\end{equation*}

Moreover, there exists an open neighborhood $U$ of $\Lambda $ such that $%
k_{\varepsilon }(\mathcal{N})=\Lambda _{\varepsilon }\subset U$ is a
normally hyperbolic invariant manifold and%
\begin{equation*}
\Lambda _{\varepsilon }=\bigcap_{n\in \mathbb{Z}}f_{\varepsilon }^{n}(U).
\end{equation*}
\end{theorem}

There are more general results on the
persistence of normally hyperbolic invariant manifolds that are not compact,
or of manifolds with boundary (see \cite%
{Bates99,Berger13,DLS3,Eldering12,Marco16}). In the case of manifolds with
boundary, the manifold that persists under the perturbation is, in general,
only locally invariant. The proof in that case involves extending the vector
field in such a way that the manifold we consider is an invariant manifold
without boundary. Then, applying the result of persistence of an invariant
manifold without boundary, one obtains the existence of a locally invariant
manifold. It is important to point out that the manifold thus produced is
not unique, as it depends on the extension considered. Also, when discussing
stable/unstable manifolds and fibres of the perturbed normally hyperbolic
invariant manifolds with boundary, we have to have in mind the
stable/unstable manifolds and fibres of the normally hyperbolic manifold
(without boundary) under the extended vector field. While the persistent
manifold is not unique, all orbits that remain in a small neighborhood of
the manifold and away from its boundary remain present in all extensions
that do not modify the dynamics in that neighborhood.

We are not aware of any method ensuring that the extended vector field is
symplectic. On the other hand, we note that perturbations of symplectic
manifolds remain symplectic because the closedness of the form is automatic
and the non-degeneracy is true because perturbations of non-degenerate form
remain non-degenerate. In particular, when considering perturbations of
invariant manifolds in symplectic systems, we obtain that they are
symplectic and that we can apply KAM theory and variational methods on them.
This is the case considered in our application. An example of this is given
in the proof of Theorem \ref{th:kam-per3bp},
where we show an example of an extension that
ensures the symplecticity within the domain where KAM is applied.

We also point out that if the locally invariant manifold was not invariant,
the goal of achieving diffusion would have been accomplished, since the
orbits which make the manifold not invariant (even after choosing a smaller
submanifold) have to move by order 1.

We also have the following:

\begin{lemma}
\cite{DLS1} \label{lem:nhim-symplectic} In the case that the map $f$
preserves a symplectic form $\omega$, we have that $\omega|_{\Lambda}$ is a
symplectic form and $f|_{\Lambda}$ preserves $\omega|_{\Lambda}$.
\end{lemma}

\begin{proof}
Since $d\omega=0$, it is clear that $d_{\Lambda}\omega|_{\Lambda}=0.$ To
conclude that $\omega_{\Lambda}$ is symplectic, it is sufficient to show
that $\omega|_{\Lambda}$ is not degenerate.

We observe that, if $c\in T_{x}\Lambda $, $s\in E_{x}^{s}$, $u\in E_{x}^{u}$%
, we have, by the preservation of the symplectic form for any $n,m\in
\mathbb{Z}$
\begin{eqnarray*}
\omega \left( c,s\right) &=&\omega \left( Df^{n}c,Df^{n}s\right) , \\
\omega \left( c,u\right) &=&\omega \left( Df^{m}c,Df^{m}u\right) .
\end{eqnarray*}%
Using the different rates (\ref{eq:rate-cond-nhim1}--\ref{eq:rate-cond-nhim3}%
) and taking limits $n\rightarrow \infty $ and $m\rightarrow -\infty $ we
obtain that%
\begin{equation*}
\omega \left( c,s\right) =\omega \left( c,u\right) =0
\end{equation*}%
Hence if $\omega |_{\Lambda }\left( c,\bar{c}\right) =0$ for all $\bar{c}\in
T_{x}\Lambda ,$ we conclude that $\omega \left( c,v\right) =0$ for all
vectors $v$. Using the nondegeneracy of $\omega $, we conclude that $c=0$.
Hence, we have shown that $\omega |_{\Lambda }$ is not degenerate.
\end{proof}

\subsection{Scattering map for normally hyperbolic invariant manifolds and
specially for Hamiltonian systems}

In this section, we review the scattering map, introduced in \cite{DLS3} to
quantify the properties of homoclinic excursions. A systematic exposition is
in \cite{DLS1}. Heuristic descriptions of its role in Arnol'd diffusion are
in \cite{DLS2,DelshamsGLS2008}.

\label{subsection:scattering} Consider a Hamiltonian system $H:\mathbb{R}%
^{2n}\rightarrow \mathbb{R}$. Let $\Phi _{t}$ denote the time shift map
along a trajectory of
\begin{equation*}
\mathbf{\dot{x}}=J\nabla H(\mathbf{x}),
\end{equation*}%
where
\begin{equation*}
J=\left(
\begin{array}{cc}
0 & -Id \\
Id & 0%
\end{array}%
\right) ,
\end{equation*}%
and $Id$ is the $n\times n$ identity matrix. Let us consider a fixed $t\in
\mathbb{R}$ and assume that $\Lambda $ is a normally hyperbolic invariant
manifold for $\Phi _{t}$, with a reference manifold $\mathcal{N}$,
parametrization $k_{0}:\mathcal{N}\rightarrow \mathbb{R}^{2n}$ and inner
dynamics $r_{0,t}:\mathcal{N}\rightarrow \mathcal{N}$,
\begin{eqnarray*}
\Lambda &=&k_{0}\left( \mathcal{N}\right) , \\
\Phi _{t}\circ k_{0} &=&k_{0}\circ r_{0,t}.
\end{eqnarray*}

Let us define two maps, which we refer to as the wave maps
\begin{eqnarray*}
\Omega _{+}& :W^{s}(\Lambda )\rightarrow \Lambda , \\
\Omega _{-}& :W^{u}\left( \Lambda \right) \rightarrow \Lambda ,
\end{eqnarray*}%
where $\Omega _{+}(\mathbf{x})=\mathbf{x}_{+}$ iff $\mathbf{x}\in
W^{s}\left( \mathbf{x}_{+}\right) $, and $\Omega _{-}(\mathbf{x})=\mathbf{x}%
_{-}$ iff $\mathbf{x}\in W^{u}\left( \mathbf{x}_{-}\right) .$

\begin{definition}
\label{def:homoclinic-channel} We say that a manifold $\Gamma $ is a
homoclinic channel for $\Lambda $ if the following conditions hold:

\begin{itemize}
\item[(i)] \label{itm:homoclinic-channel-c1} for every $\mathbf{x}\in \Gamma
$
\begin{eqnarray*}
T_{\mathbf{x}}W^{s}\left( \Lambda \right) \oplus T_{\mathbf{x}}W^{u}\left(
\Lambda \right) & =\mathbb{R}^{2n}, \\
T_{\mathbf{x}}W^{s}\left( \Lambda \right) \cap T_{\mathbf{x}}W^{u}\left(
\Lambda \right) & =T_{\mathbf{x}}\Gamma .
\end{eqnarray*}

\item[(ii)] \label{itm:homoclinic-channel-c2} the fibres of $\Lambda $
intersect $\Gamma $ transversally in the following sense%
\begin{eqnarray*}
T_{\mathbf{x}}\Gamma \oplus T_{\mathbf{x}}W^{s}\left( \mathbf{x}_{+}\right)
& =&T_{\mathbf{x}}W^{s}\left( \Lambda \right) , \\
T_{\mathbf{x}}\Gamma \oplus T_{\mathbf{x}}W^{u}\left( \mathbf{x}_{-}\right)
& =&T_{\mathbf{x}}W^{u}\left( \Lambda \right) ,
\end{eqnarray*}%
for every $\mathbf{x}\in \Gamma $,

\item[(iii)] the wave maps $(\Omega _{\pm })_{\mid \Gamma }:\Gamma
\rightarrow \Lambda $ are diffeomorphisms.
\end{itemize}
\end{definition}

\begin{definition}
Assume that $\Gamma$ is a homoclinic channel for $\Lambda$ and let
\begin{equation*}
\Omega_{\pm}^{\Gamma}:=\left( \Omega_{\pm}\right) |_{\Gamma}.
\end{equation*}
We define a scattering map $\sigma^{\Gamma}$ for the homoclinic channel $%
\Gamma$ as
\begin{equation*}
\sigma^{\Gamma}:=\Omega_{+}^{\Gamma}\circ\left( \Omega_{-}^{\Gamma}\right)
^{-1}:\Omega_{-}^{\Gamma}\left( \Gamma\right) \rightarrow\Omega_{+}^{\Gamma
}\left( \Gamma\right) .
\end{equation*}
\end{definition}

Two important properties of the scattering map will be used later.

First is the symplectic property of the scattering map. Let $\omega$ stand
for the standard symplectic form in $\mathbb{R}^{2n}$. If $%
\omega_{\mid\Lambda}$ is also symplectic, then the scattering map $%
\sigma^\Gamma$ is symplectic.

Second is an invariance property of the scattering map. Note that if $\Gamma
$ is a homoclinic channel, then for each $T$, $\Phi _{T}(\Gamma )$ is also a
homoclinic channel. The corresponding scattering map $\sigma ^{\Phi_T
(\Gamma )}$ is related to $\sigma ^{\Gamma }$ by the following relation
\begin{equation}
\sigma ^{\Phi _{T}(\Gamma )}=\Phi _{T}\circ \sigma ^{\Gamma }\circ \Phi
_{-T}.  \label{eq:invariance-property}
\end{equation}%
This says that, while $\sigma ^{\Gamma }$ and $\sigma ^{\Phi _{T}(\Gamma )}$
are technically different scattering maps, they are nevertheless conjugated
via the flow.

Let us consider a family of $C^r$ Hamiltonians $H_{\varepsilon}:\mathbb{R}%
^{2n}\rightarrow\mathbb{R}$, with $r\ge 2$, depending smoothly on $%
\varepsilon$, such that
\begin{equation*}
H_{0}=H.
\end{equation*}
Let $\Phi_{\varepsilon,t}$ stand for the time $t$ shift along a trajectory
of
\begin{equation*}
\mathbf{\dot{x}}=J\nabla H_{\varepsilon}(\mathbf{x}).
\end{equation*}
By the normally hyperbolic invariant manifold theorem (Theorem \ref%
{th:nhim-pert}) $\Lambda$ is perturbed to $\Lambda_{\varepsilon}$, a
normally hyperbolic invariant manifold for $\Phi_{\varepsilon,t}$.

The following theorem gives us a parametrization of $\Lambda_{\varepsilon},$
which preserves the symplectic form.

\begin{theorem}
\cite[Theorems 23,24,25]{DLS1}\label{th:k-epsilon} Assume that $\omega
|_{\Lambda}$ is non degenerate and let $\omega_{\mathcal{N}}:=k_{0}^{\ast
}\omega|_{\Lambda}$ . Then there exist an $\varepsilon_0>0$, such that if $%
H_{\varepsilon}$ are $\varepsilon_0$ close in $C^r$-topology, we have a
smooth family of maps $k_{\varepsilon }:\mathcal{N}\rightarrow\mathbb{R}%
^{2n} $, and a smooth family of flows $r_{\varepsilon,t}:\mathcal{N}\times%
\mathbb{R} \rightarrow\mathcal{N}$, such that%
\begin{equation*}
\Phi_{\varepsilon,t}\circ k_{\varepsilon}=k_{\varepsilon}\circ
r_{\varepsilon,t},
\end{equation*}
and%
\begin{equation*}
\Lambda_{\varepsilon}=k_{\varepsilon}\left( \mathcal{N}\right)
\end{equation*}
is a normally hyperbolic invariant manifold for $\Phi_{\varepsilon,t}.$
Moreover
\begin{equation*}
k_{\varepsilon}^{\ast}\omega|_{\Lambda_{\varepsilon}}=\omega_{\mathcal{N}},
\end{equation*}
is independent of $\varepsilon$, and
\begin{equation*}
r_{\varepsilon,t}^{\ast}\omega_{\mathcal{N}}=\omega_{\mathcal{N}},
\end{equation*}
for all $t$.
\end{theorem}

Transverse intersections of stable/unstable manifolds are robust under
perturbation. This means that the homoclinic channel $\Gamma $ is perturbed
to a homoclinic channel $\Gamma _{\varepsilon }$ for $\Phi _{\varepsilon
,t}. $ This leads to a scattering map map for $\Phi _{\varepsilon ,t}$
\begin{equation*}
\sigma _{\varepsilon }^{\Gamma _{\varepsilon }}:\Omega _{-}^{\Gamma
_{\varepsilon }}\left( \Gamma _{\varepsilon }\right) \rightarrow \Omega
_{+}^{\Gamma _{\varepsilon }}\left( \Gamma _{\varepsilon }\right) .
\end{equation*}%
It is convenient to express the scattering map as a map on the reference
manifold $\mathcal{N}$, by defining%
\begin{equation*}
s_{\varepsilon }=k_{\varepsilon }^{-1}\circ \sigma _{\varepsilon }\circ
k_{\varepsilon },
\end{equation*}%
\begin{equation*}
s_{\varepsilon }:\mathcal{N}\supset k_{\varepsilon }^{-1}\circ \Omega
_{-}^{\Gamma _{\varepsilon }}\left( \Gamma _{\varepsilon }\right)
\rightarrow k_{\varepsilon }^{-1}\circ \Omega _{+}^{\Gamma _{\varepsilon
}}\left( \Gamma _{\varepsilon }\right) \subset \mathcal{N}.
\end{equation*}%
Below we give a diagram which summarizes all the maps involved in the
definition%
\begin{equation*}
\begin{array}{ccc}
\Lambda _{\varepsilon } &
\begin{array}{ccc}
\overset{\left( \Omega _{-}^{\Gamma _{\varepsilon }}\right) ^{-1}}{%
\longrightarrow } & \Gamma _{\varepsilon } & \overset{\Omega _{+}^{\Gamma
_{\varepsilon }}}{\longrightarrow }%
\end{array}
& \Lambda _{\varepsilon } \\
\quad \uparrow k_{\varepsilon } &  & \quad \uparrow k_{\varepsilon } \\
\mathcal{N} & \underrightarrow{\quad \quad s_{\varepsilon }\quad } &
\mathcal{N}%
\end{array}%
\end{equation*}

\begin{theorem}
\label{th:Melnikov-potential}\cite[Theorem 32]{DLS1} The map $s_{\varepsilon
}$ is a symplectic map. Moreover,%
\begin{equation*}
s_{\varepsilon }=s_{0}+\varepsilon J\nabla S_{0}\circ s_{0}+O\left(
\varepsilon ^{2}\right),
\end{equation*}%
with%
\begin{eqnarray*}
S_{0}\left( \mathbf{x}\right) =&\lim_{T\rightarrow +\infty }\int_{-T}^{0}%
\frac{dH_{\varepsilon }}{d\varepsilon }|_{\varepsilon =0}\circ \Phi
_{u}\circ \left( \Omega _{-}^{\Gamma }\right)^{-1}\circ (\sigma^\Gamma)^{-1}
\circ k_{0}(\mathbf{x)} \\
& \qquad \qquad \qquad -\frac{dH_{\varepsilon }}{d\varepsilon }%
|_{\varepsilon =0}\circ \Phi _{u}\circ \left( \sigma ^{\Gamma }\right)
^{-1}\circ k_{0}(\mathbf{x)}du \\
& +\lim_{T\rightarrow +\infty }\int_{0}^{T}\frac{dH_{\varepsilon }}{%
d\varepsilon }|_{\varepsilon =0}\circ \Phi _{u}\circ \left( \Omega
_{+}^{\Gamma }\right) ^{-1}\circ k_{0}(\mathbf{x)} \\
& \qquad \qquad \qquad -\frac{dH_{\varepsilon }}{d\varepsilon }%
|_{\varepsilon =0}\circ \Phi _{u}\circ k_{0}(\mathbf{x})du.
\end{eqnarray*}
\end{theorem}

\subsection{KAM theorem}

The celebrated KAM theorem is used to prove persistence of invariant tori.
We focus on the setting of a symplectic map on an annulus, since such will
be the setting in the restricted three body problem.

Let $\mathbb{T}^{1}=\mathbb{R}/2\pi\mathbb{Z}$ be a circle.

\begin{theorem}[KAM Theorem]
\cite[Theorem 4.8]{DLS3}\label{th:KAM} Let $g:\left[ 0,1\right] \times
\mathbb{T}^{1}\rightarrow \left[ 0,1\right] \times \mathbb{T}^{1}$ be an
exact symplectic $C^{l}$ map with $l\geq 6.$ Assume that $%
g=g_{0}+\varepsilon g_{1}$, where $\varepsilon \in \mathbb{R},$
\begin{equation}
g_{0}\left( I,\varphi \right) =\left( I,\varphi +A\left( I\right) \right) ,
\label{eq:KAM-form}
\end{equation}%
$A$ is $C^{l},$ $\left\vert \frac{dA}{dI}\right\vert \geq M$, and $%
\left\Vert g_{1}\right\Vert _{C^{l}}\leq 1.$ Then, for each $\varepsilon$ 
 sufficiently small, for a set of Diophantine frequencies $%
\sigma $ of exponent $\theta =5/4$ \footnote{$\sigma $ is a Diophantine
number of exponent $\theta $ if there exists $C>0$ such that $|\sigma
-p/q|>C/q^{\theta +1}$ for all $p,q\in \mathbb{Z}$ with $q\neq 0$}, there exist invariant tori which are graphs of $C^{l-3}$ functions $u_{\sigma
}=u_{\sigma }(\varphi )$, the motion on them is $C^{l-3}$-conjugate to the
rotation by $\sigma $, and the tori cover the whole annulus except for a set of
measure smaller than $O\left( M^{-1}\varepsilon ^{1/2}\right) $.%

\end{theorem}

\section{Diffusion mechanism for time periodic perturbations of Hamiltonian
systems\label{sec:diffusion-mechanism}}

We now consider a particular formulation of Theorem \ref%
{th:Melnikov-potential}. Let%
\begin{equation}
H_{\varepsilon}(\mathbf{x},t)=H(\mathbf{x})+\varepsilon G(\mathbf{x}%
,t)+O(\varepsilon^{2})  \label{eq:pert-Ham}
\end{equation}
with $G$ being $2\pi$ periodic in $t$. We assume that $H_{\varepsilon}$
depend smoothly on $\varepsilon$.

Let us assume that for $\varepsilon=0$ we have a normally hyperbolic
invariant manifold $\Lambda$ for $\Phi_{2\pi}$, with a reference manifold $%
\mathcal{N}$, parametrization $k_{0}$ and inner dynamics $r_{0}$.

Assume that there exists a finite collection of homoclinic channels $%
\Gamma^j $, $j=1,\ldots,k$, for $\Phi_{2\pi}$, and corresponding scattering
maps $\sigma^{\Gamma^j}:\Omega^{\Gamma^j}_-(\Gamma^j)\to
\Omega^{\Gamma^j}_+(\Gamma^j)$, $j=1,\ldots,k$. Each scattering map can be
expressed as a map on the reference manifold $\mathcal{N}$ by $%
s^j_{0}=k_{0}^{-1}\circ\sigma^{\Gamma^j}\circ k_{0}$, $j=1,\ldots,k$.

In what follows, we will switch from studying the flow dynamics to the
dynamics induced by a time-$2\pi$ map of the flow. We note that the
scattering maps for the flow $\Phi_t$ from the above collection remain
scattering maps for the time-$2\pi$ map of the flow (see \cite{DLS1}).

Consider now $\varepsilon >0$. Let $\Sigma _{t=\tau }=\left\{ \left( \mathbf{%
x},t\right) |t=\tau \right\} $ and $\Phi _{\varepsilon ,\tau ,2\pi }:\Sigma
_{t=\tau }\rightarrow \Sigma _{t=\tau }$ be the map induced by the time $%
2\pi $ shift along the flow of $H_{\varepsilon }$.

We assume that the manifold $\Lambda $ for $\varepsilon =0$ is perturbed to $%
\Lambda _{\varepsilon ,\tau }$, which is a normally hyperbolic invariant
manifold for $\Phi _{\varepsilon ,\tau ,2\pi }$, provided $\varepsilon $ is
sufficiently small. This is parametrized by $k_{\varepsilon ,\tau }:\mathcal{%
N}\rightarrow \Lambda _{\varepsilon ,\tau }$. We denote by $r_{\varepsilon
,\tau ,2\pi }$ the map induced by $\Phi _{\varepsilon ,\tau ,2\pi }$ on the
reference manifold $\mathcal{N}$, i.e., 
\begin{equation*}
r_{\varepsilon ,\tau ,2\pi }=k_{\varepsilon ,\tau }^{-1}\circ \Phi
_{\varepsilon ,\tau ,2\pi }\circ k_{\varepsilon ,\tau }.
\end{equation*}%
By the previous section, $\Phi _{\varepsilon ,\tau ,2\pi }$ is a symplectic
map on $\Lambda _{\varepsilon ,\tau }$, and $r_{\varepsilon ,\tau ,2\pi }$
is a symplectic map on $\mathcal{N}$.

Also, the homoclinic channels $\Gamma^j$ are perturbed to $%
\Gamma^j_{\varepsilon,\tau}$, $j=1,\ldots,k$, respectively, leading to
scattering maps $\sigma^{\Gamma^j}_{\varepsilon,\tau}$, and to the
corresponding maps defined on the reference manifold $\mathcal{N}$ 
\begin{equation*}
s^j_{\varepsilon,\tau}=k_{\varepsilon,\tau}^{-1}\circ\sigma^{\Gamma^j
_{\varepsilon,\tau}}\circ k_{\varepsilon,\tau}.
\end{equation*}
From the previous section we have that each map $s^j_{\varepsilon,\tau}$, $%
j=1,\ldots,k$, is symplectic.

Again, the advantage of expressing a scattering map in terms of the
reference manifold is that the unperturbed scattering map as well as its
sufficiently small perturbations, are defined on some domains of the same
manifold $\mathcal{N}$.

For a generic map $s^{\Gamma}_{\varepsilon,\tau}$ from this family we have
the following result:

\begin{theorem}
\label{th:Melnikov-potential-nonaut} For $\varepsilon$ sufficiently small,
so that the scattering map is well defined, 
\begin{equation}
s_{\varepsilon,\tau}=s_{0}+\varepsilon J\nabla S_{0,\tau}\circ s_{0}+O\left(
\varepsilon^{2}\right) ,  \label{eq:s-eps-nonaut}
\end{equation}
with 
\begin{eqnarray}
S_{0,\tau}\left( \mathbf{x}\right) & =\lim_{T\rightarrow+\infty}\int
_{-T}^{0}G\left( \Phi_{u}\circ\left( \Omega_{-}^{\Gamma}\right)
^{-1}\circ(\sigma^\Gamma)^{-1} \circ k_{0}(\mathbf{x),}\tau+u\right)
\label{eq:S0tau-def} \\
& \qquad\qquad\qquad-G\left( \Phi_{u}\circ\left( \sigma^{\Gamma}\right)
^{-1}\circ k_{0}(\mathbf{x),}\tau+u\right) du  \nonumber \\
& +\lim_{T\rightarrow+\infty}\int_{0}^{T}G\left( \Phi_{u}\circ\left(
\Omega_{+}^{\Gamma}\right) ^{-1}\circ k_{0}(\mathbf{x)},\tau+u\right)  \nonumber
\\
& \qquad\qquad\qquad-G\left( \Phi_{u}\circ k_{0}(\mathbf{x),}\tau+u\right)
du.  \nonumber
\end{eqnarray}
\end{theorem}

\begin{proof}
The proof follows by applying Theorem \ref{th:Melnikov-potential} in the
extended phase space. 
\end{proof}

Assume that $\mathcal{N}=\left[ 0,1\right] \times\mathbb{T}^{1}$. We can
describe $\mathcal{N}$ by a system of action-angle coordinates $(I,\theta)$
with $I\in[0,1]$ and $\theta\in\mathbb{T}^1$, with $dI\wedge d\theta=\omega_{%
\mathcal{N}}$.

We assume that the unperturbed scattering maps $s_{0}^{1},\ldots ,s_{0}^{k}$
from above satisfy 
\begin{eqnarray}
s_{0}^{j}& :\mathcal{N}\supset \mathrm{dom}(s_{0}^{j})\rightarrow \mathrm{ran%
}(s_{0}^{j})\subset \mathcal{N},  \nonumber \\
s_{0}^{j}(I,\theta )& =\left( I,\omega ^{j}(I,\theta )\right) ,
\label{eq:s-rotation-form}
\end{eqnarray}%
for $j=1,\ldots ,k$. The domain $\mathrm{dom}(s_{0}^{j})$ and the range $%
\mathrm{ran}(s_{0}^{j})$ of each unperturbed scattering map are assumed to
be open sets in $\mathcal{N}$. Note that each unperturbed scattering map $%
s_{0}^{j}$, as well as each perturbation $s_{\varepsilon ,\tau }^{j}$, is an
area preserving map on $\mathcal{N}$.

The assumption (\ref{eq:s-rotation-form}) is natural for several types of
systems. In the model for the large gap problem considered in \cite{DLS2},
it is shown that there exists a scattering map of the form $%
s_0(I,\theta)=(I,\theta)$, that is, $s_0=\text{Id}$. In the periodically
perturbed geodesic flow model considered in \cite{DLS3}, it is shown that
there exists a scattering map of the form $s_0(I,\theta)=(I,\theta+\psi)$,
for some constant $\psi$. In the PER3BP model considered in this paper, we
find that there exist scattering maps of the form $s_{0} (I,\theta) =\left(
I,\theta+\omega (I)\right)$; see Section \ref{sec:3bp}. Such form of the
scattering map has been established and numerically investigated in \cite%
{Canalias}.

We also assume that the unperturbed inner map $r_{0}$, as well as each
perturbation $r_{\varepsilon ,\tau }$, is an area preserving map on $%
\mathcal{N}$.

We make the following assumption:%
\begin{equation}
\text{int}(\mathcal{N})=(0,1)\times \mathbb{T}^{1}\subseteq
\bigcup_{j=1,\ldots ,k}\mathrm{dom}(s_{0}^{j}).
\label{eq:scattering-domains}
\end{equation}
This assumption means that for every point $(I,\theta )\in \text{int}(%
\mathcal{N})$ there exists a scattering map $s_{0}^{j}$, for some $j$,
defined on a neighborhood of that point. It may seem as a very strong
assumption. However, the invariance property of the scattering map,
mentioned in Subsection \ref{subsection:scattering}, implies that if $%
s_{0}^{j}:\text{dom}(s_{0}^{j})\rightarrow \text{ran}(s_{0}^{j})$ is a
scattering map corresponding to a homoclinic channel $\Gamma ^{j}$, then the
scattering map corresponding to the homoclinic channel $f_{0,\tau
}^{k}(\Gamma ^{j})$ is defined on $r_{0,\tau }^{k}(\text{dom}%
(s_{0}^{j}))\subseteq \mathcal{N}$, for all $k\in \mathbb{Z}$. Hence, simply
by iterating the homoclinic channel we obtain corresponding scattering maps
whose domain in $\mathcal{N}$ is iterated by the inner dynamics. Thus we can
cover large portions of $\mathcal{N}$ with domains of scattering maps just
using the invariance property from above. This idea will be illustrated in
Section \ref{sec:3bp}.

Denote%
\begin{eqnarray*}
\mathcal{N}^{<}\left( a\right) & :=\left\{ (I,\theta )\in \mathcal{N}%
|I<a\right\} , \\
\mathcal{N}^{>}\left( a\right) & :=\left\{ (I,\theta )\in \mathcal{N}%
|I>a\right\} ,
\end{eqnarray*}%
and for $\rho >0$ denote 
\begin{eqnarray*}
B_{\rho }^{<}(a) &=&\{\mathbf{x}:d\left( \mathbf{x},k_{0}\left( \mathcal{N}%
^{<}\left( a\right) \right) \right) <\delta \}, \\
B_{\rho }^{>}(a) &=&\{\mathbf{x}:d\left( \mathbf{x},k_{0}\left( \mathcal{N}%
^{>}\left( a\right) \right) \right) <\delta \}.
\end{eqnarray*}

Let $\tau\in\lbrack0,2\pi)$ be a fixed number. Let $S_{0,\tau}^{j}$ stand
for the functions of the form (\ref{eq:S0tau-def}), associated to the
perturbed scattering maps $s_{\varepsilon,\tau}^{j}$, respectively, for $%
j=1,\ldots,k$.

We now state the main theoretical result that we will use for our proof of
diffusion.

\begin{theorem}
\label{th:mechanism-main} Consider that all above mentioned assumptions from
this section are fulfilled. In particular, that for $\varepsilon=0$ we have
the sequence of scattering maps $s_{0}^{1},\ldots,s_{0}^{k}$ of the form (%
\ref{eq:s-rotation-form}), satisfying (\ref{eq:scattering-domains}).

Let $\rho >0$ be a fixed number. If for every $\left( I,\theta \right) \in 
\mathcal{N}$ there exists $j$ such that 
\begin{equation}
\frac{\partial S_{0,\tau }^{j}}{\partial \theta }( s_0^j(I,\theta))<0,
\label{eq:scatter-grad-a1}
\end{equation}%
then for every $a_{1}<a_{2}$ in $(0,1)$, and for all sufficiently small $%
\varepsilon $, there exists an orbit from $B_{\rho }^{<}(a_1)$ to $B_{\rho
}^{>}(a_2)$.

Similarly, if for every $\left( I,\theta \right) \in \mathcal{N}$ there
exists $j$ such that 
\begin{equation}
\frac{\partial S_{0,\tau }^{j}}{\partial \theta }( s_0^j(I,\theta ))>0,
\label{eq:scatter-grad-a2}
\end{equation}%
then for every $a_{1}<a_{2}$ in $(0,1)$, and for all sufficiently small $%
\varepsilon $ there exists an orbit from $B_{\rho }^{>}(a_2)$ to $B_{\rho
}^{<}(a_1)$.
\end{theorem}

To prove Theorem \ref{th:mechanism-main}, we will use a slight modification
of a shadowing-type of result from \cite{GideaLlaveSeara14}.

\begin{theorem}
\label{thm:GLS} Assume that $F:M\to M$ is a symplectic $C^r$-diffeomorphism, 
$r\geq r_0$ \footnote[1]{%
The arguments done in detail in this paper seem to require only $r_0 \ge 2$
(or even 1).  In the proof of the result, \cite{GideaLlaveSeara14} use other
papers that require to take $r_0 = 3$. It is possible that the arguments in
those papers can be improved to smaller regularity requirements and this
will lower the value of $r_0$. Of course, in applications to celestial
mechanics, regularity is not an issue} , on a symplectic, compact manifold $M
$, and $\Lambda\subseteq M$ is a compact, normally hyperbolic invariant
manifold for $F$ which is also symplectic.

Also assume that there exists a finite family of homoclinic channels $%
\Gamma^j\subseteq M$, with corresponding scattering maps $\sigma^{\Gamma_j}$%
, for $j=1,\ldots, k$.

Let $\{x_i\}_{i=0,\ldots,n}$ be a sequence of points in $\Lambda$ obtained
by successive applications of scattering maps from the given family 
\begin{equation*}
x_{i+1}=\sigma^{\Gamma_\alpha(i)}(x_i)\text{ for some } \alpha(i)\in\{1,%
\ldots, k\},
\end{equation*}
for $i=0,\ldots,n-1$.

Then, for every $\delta>0$ there exist an orbit $\{z_i\}_{i=0,\ldots,n}$ of $%
F$ in $M$, with $z_{i+1}=F^{k_i}(z_i)$ for some $k_i>0$, such that $%
d(z_i,x_i)<\delta$ for all $i=0,\ldots,n$.
\end{theorem}

\begin{remark}
An analogous shadowing result to Theorem \ref{thm:GLS} can also be found in 
\cite{GT}.
\end{remark}

We emphasize that in the statement of Theorem \ref{thm:GLS} no conditions
are required on the inner dynamics. In the particular case when $\Lambda $
is an annulus, it is not required, for example, that $F$ restricted to the
annulus satisfies a twist condition, which is a standard condition in many
similar shadowing types of results.

The only property of the inner dynamics that is used in Theorem \ref{thm:GLS}
is that almost every point in $\Lambda $ is recurrent for $F_{\mid \Lambda }$%
. A point $x\in \Lambda $ is recurrent if for every neighborhood $V$ of $x$
in $\Lambda $, $F^{n}(x)\in V$ for some $n>0$. The Poincar\'{e} recurrence
theorem states that for a measure preserving map on a finite measure space,
a.e. point is recurrent. In the situation described by Theorem \ref{thm:GLS}%
, $F_{\mid \Lambda }$ is symplectic hence it preserves the volume form on $%
\Lambda $, and $\Lambda $ is compact hence of finite measure. Thus, by
Poincar\'{e} recurrence theorem, a.e. point in $\Lambda $ is recurrent for $%
F_{\mid \Lambda }$.

The times $k_i$ that appear in Theorem \ref{thm:GLS} depend on the choice of 
$\delta$, on the angle of intersection between $W^u(\Lambda)$ and $%
W^s(\Lambda)$ along the homoclinic channels $\Gamma^j$, and on the
ergodization time of the inner dynamics, i.e., the dynamics of $F$
restricted to $\Lambda$. The ergodization time can be quantitatively
estimated in some cases, for example, if the inner map is a small
perturbation of an integrable twist map.

Theorem \ref{thm:GLS} can be proved by using the method of correctly aligned
windows which is constructive. Thus, the existence of trajectories resulting
from this mechanism can be implemented in rigorous numerical arguments.

\begin{proof}[Proof of Theorem \protect\ref{th:mechanism-main}]
We will prove only the first statement of Theorem \ref{th:mechanism-main},
as the second one follows similarly.

Fix $0<a_{1}<a_{2}<1$. To apply Theorem \ref{thm:GLS} we have to produce a
sequence $\{x_{i}\}_{i=0,\ldots ,n}$ obtained by successively applying some
scattering map $\sigma _{\varepsilon ,\tau }^{\Gamma_j}$ at each step, with $%
I(x_{0})<a_{1}$ and $I(x_{n})>a_{2}$. We have that $[a_{1},a_{2}]\times 
\mathbb{T}^{1}\subseteq \bigcup_{j=1,\ldots ,k}\mathrm{dom}(s_{0}^{j})$.
Since each domain $\mathrm{dom}(s_{0}^{j})$ is an open set, by the
continuous dependence of $s_{\varepsilon ,\tau }^{j}$ on $\varepsilon $,
there exists $\varepsilon _{0}$ such that, for all $0<\varepsilon
<\varepsilon _{0}$ we have $[a_{1},a_{2}]\times \mathbb{T}^{1}\subseteq
\bigcup_{j=1,\ldots ,k}\mathrm{dom}(s_{\varepsilon ,\tau }^{j})$.

By compactness, the assumption that for every $(I,\theta )$ there exists $j$
with $\frac{\partial S_{0,\tau }^{j}}{\partial \theta }(s_{0}^{j}\left(
I,\theta \right) )<0$ implies that there exists $c>0$ such that for every $%
(I,\theta )$ there is a $j$ with $\frac{\partial S_{0,\tau }^{j}}{\partial
\theta }(s_{0}^{j}\left( I,\theta \right) )<-c$.

By (\ref{eq:s-eps-nonaut}) and (\ref{eq:s-rotation-form}) we have that 
\begin{eqnarray*}
I(s_{\varepsilon ,\tau }^{j})(I,\theta ) &=&I-\varepsilon \frac{\partial
S_{0,\tau }^{j}}{\partial \theta }(s_{0}^{j}\left( I,\theta \right)
)+O(\varepsilon ^{2}), \\
\theta (s_{\varepsilon ,\tau }^{j})(I,\theta ) &=&\omega ^{j}(I,\theta
)+\varepsilon \frac{\partial S_{0,\tau }^{j}}{\partial I}(s_{0}^{j}\left(
I,\theta \right) )+O(\varepsilon ^{2}),
\end{eqnarray*}%
for all $j=1,\ldots ,k$ and all $\varepsilon \in (0,\varepsilon _{0})$,
where by $I(s_{\varepsilon ,\tau }^{j})$, $\theta (s_{\varepsilon ,\tau
}^{j})$ we denote the $I$- and $\theta $-components of $s_{\varepsilon ,\tau
}^{j}$, respectively.

Since $\frac{\partial S_{0,\tau }^{j}}{\partial \theta }( s_{0}^{j}(
I,\theta)) <-c$, this implies, again for $\varepsilon _{0}$ small enough and
all $\varepsilon \in (0,\varepsilon _{0})$, that for every $(I_{a},\theta
_{a})$ there exists a $j=j(a)$ such that $s_{\varepsilon ,\tau
}^{j(a)}(I_{a},\theta _{a})=(I_{b},\theta _{b})$, with $I_{b}-I_{a}>c%
\varepsilon $.

Thus, choosing an initial point $(I_0,\theta_0)$ with $I_0<a_1$, we can
construct a sequence of points $(x_i)_{i=0,\ldots,n}\subseteq \mathcal{N}$,
with $x_i=(I_i,\theta_i)$ and $n=O(1/\varepsilon)$, such that $%
s^{j(i)}_{\varepsilon,\tau} (I_i,\theta_i)=(I_{i+1},\theta_{i+1})$, and $%
I_{i+1}-I_{i}>c\varepsilon$, for all $i=0,\ldots,n-1$, and $I_n>a_2$.

Now we consider the corresponding sequence of points in $\Lambda_{%
\varepsilon,\tau}$, obtained via the parametrization $k_{\varepsilon,\tau}$.
Let $y_i=k_{\varepsilon,\tau}(x_i)\in \Lambda_{\varepsilon,\tau}$, for $%
i=0,\ldots,n$. By the relation between $\sigma^j_{\varepsilon,\tau} $ and $%
s^{j}_{\varepsilon,\tau}$, we have that $y_{i+1}=\sigma^{j(i)}_{\varepsilon,%
\tau}(y_i)$, for $i=0,\ldots,n$. We recall that each $\sigma^{j(i)}_{%
\varepsilon,\tau}$ is an area preserving map on $\Lambda_{\varepsilon,\tau}$%
, and also that $\Phi_{\varepsilon,\tau,2\pi}$ is an area preserving map on $%
\Lambda_{\varepsilon,\tau}$.

By the smooth dependence of $k_{\varepsilon ,\tau }$ on $\varepsilon $ and
the compactness of $\mathcal{N}$, if $\varepsilon _{0}$ is small enough,
then for all $\varepsilon \in (0,\varepsilon _{0})$ and all $(I,\theta)\in%
\mathcal{N}$ we have 
\begin{eqnarray*}
d\left( k_{\varepsilon ,\tau }(I,\theta ),k_{0}(I,\theta )\right) &<&\rho /2.
\end{eqnarray*}
Let $\delta =\rho /2$. Theorem \ref{thm:GLS} implies that there exists an
orbit $(z_{i})_{i=0,\ldots ,n}$ of $\Phi _{\varepsilon ,\tau ,2\pi }$ such
that $z_{i+1}=\Phi _{\varepsilon ,\tau ,2\pi }^{k_{i}}(z_{i})$, for some $%
k_{i}>0$, and with $d(z_{i},y_{i})<\delta =\rho /2$, for all $i=0,\ldots ,n$%
. This means that%
\begin{eqnarray*}
d(z_{0},k_{0}\left( \mathcal{N}^{<}\left( a_{1}\right) \right) ) &\leq
&d\left( z_{0},k_{0}\left( I_{0},\theta _{0}\right) \right) \\
&\leq &d\left( z_{0},k_{\varepsilon ,\tau }(I_{0},\theta _{0})\right)
+d(k_{\varepsilon ,\tau }(I_{0},\theta _{0}),k_{0}\left( I_{0},\theta
_{0}\right) ) \\
&=&d\left( z_{0},y_{0}\right) +d(k_{\varepsilon ,\tau }(I_{0},\theta
_{0}),k_{0}\left( I_{0},\theta _{0}\right) ) \\
&<&\rho ,
\end{eqnarray*}
hence $z_{0}\in B_{\rho }^{<}(a_{1}).$ Analogous computation leads to $%
z_{n}\in B_{\rho }^{>}(a_{2})$. The orbit $(z_{i})_{i=0,\ldots ,n}$ is thus
a homoclinic orbit of the map $\Phi _{\varepsilon ,\tau ,2\pi }$ between $%
B_{\rho }^{<}(a_{1})$ and $B_{\rho }^{>}(a_{2})$, as claimed in the
statement.
\end{proof}


\section{Diffusion in the restricted three body problem}

\label{sec:3bp}

In this section we give an application of the diffusion mechanism from
section \ref{sec:diffusion-mechanism}. The existence of diffusing orbits
will result from perturbing the planar circular restricted three body
problem (PCR3BP) to the planar elliptic restricted three body problem
(PER3BP). The discussion contained in this section combines an analytical
argument with a numerical one. For the analytical part, we show how the
scattering maps can be chosen, and formulate a theorem (Theorem \ref%
{th:diffusion-3bp}) which ensures diffusion under appropriate assumptions.
In section \ref{sec:num} we give numerical verification of Theorem \ref%
{th:diffusion-3bp}.

We believe that using rigorous computer assisted computations one can obtain
a proof of diffusion using our mechanism. This will be a subject of
forthcoming work. The assumptions that would need to be checked are listed
in section \ref{sec:future-work}. They require:

\begin{enumerate}
\item[1)] a computer assisted proof of transversal intersections of
manifolds and rigorous enclosures of homoclinic orbits in the PCR3BP,

\item[2)] a rigorous enclosures of integrals along homoclinic orbits.
\end{enumerate}

Results very similar to 1) are in \cite{Ca1} and very similar to 2) are in
\cite{CZ}.

\subsection{Planar circular restricted three body problem}

In the PCR3BP we consider the motion of an infinitesimal body under the
gravitational pull of two larger bodies (which we shall refer to as
primaries) of mass $\mu$ and $1-\mu$. The primaries move around the origin
on circular orbits of period $2\pi$ on the same plane as the infinitesimal
body. In this paper we consider the mass parameter $\mu=0.0009537$, which
corresponds to the rescaled mass of Jupiter in the Jupiter-Sun system.

The Hamiltonian of the problem is given by (see \cite{AM})
\begin{equation}
H(q,p,t)=\frac{p_{1}^{2}+p_{2}^{2}}{2}-\frac{1-\mu}{r_{1}(t)}-\frac{\mu}{%
r_{2}(t)},  \label{eq:H-non-rot}
\end{equation}
where $\left( p,q\right) =\left( q_{1},q_{2},p_{1},p_{2}\right) $ are the
coordinates and momenta of the infinitesimal body relative to the center of
mass of the primaries, and $r_{1}(t)$ and $r_{2}(t)$ are the distances from
the masses $1-\mu$ and $\mu$, respectively.

After introducing a new coordinate system $(x,y,p_{x},p_{y})$%
\begin{equation}
\begin{array}{ll}
x=q_{1}\cos t+q_{2}\sin t, & \quad p_{x}=p_{1}\cos t+p_{2}\sin t, \\
y=-q_{1}\sin t+q_{2}\cos t, & \quad p_{y}=-p_{1}\sin t+p_{2}\cos t,%
\end{array}
\label{eq:x,y-coordinates}
\end{equation}
which rotates together with the primaries, the primaries become motionless
(see Figure \ref{fig:forbidden-region}) and one obtains an autonomous
Hamiltonian%
\begin{equation}
H(x,y,p_{x},p_{y})=\frac{(p_{x}+y)^{2}+(p_{y}-x)^{2}}{2}-\Omega(x,y),
\label{eq:H-PRC3BP}
\end{equation}
where%
\begin{eqnarray}
\Omega(x,y) & =\frac{x^{2}+y^{2}}{2}+\frac{1-\mu}{r_{1}}+\frac{\mu}{r_{2}},
\nonumber \\
r_{1} & =\sqrt{(x-\mu)^{2}+y^{2}},\quad r_{2}=\sqrt{(x+1-\mu)^{2}+y^{2}}.
\nonumber
\end{eqnarray}

The motion of the infinitesimal body is given by
\begin{equation}
\mathbf{\dot{x}}=J\nabla H(\mathbf{x}),  \label{eq:PRC3BP}
\end{equation}
where $\mathbf{x}=(x,y,p_{x},p_{y})\in\mathbb{R}^{4}$.

The movement of the flow (\ref{eq:PRC3BP}) is restricted to the
hyper-surfaces determined by the energy level $h$,
\begin{equation}
M(h)=\{(x,y,p_{x},p_{y})\in \mathbb{R}^{4}|H(x,y,p_{x},p_{y})=h\}.
\label{eq:M-energy}
\end{equation}%
This means that movement in the $x,y$ coordinates is restricted to the so
called Hill's region defined by
\begin{equation}
R(h)=\{(x,y)\in \mathbb{R}^{2}|\Omega (x,y)\geq -h\}.  \nonumber
\end{equation}%
\begin{figure}[tbp]
\begin{center}
\includegraphics[height=2.4in]{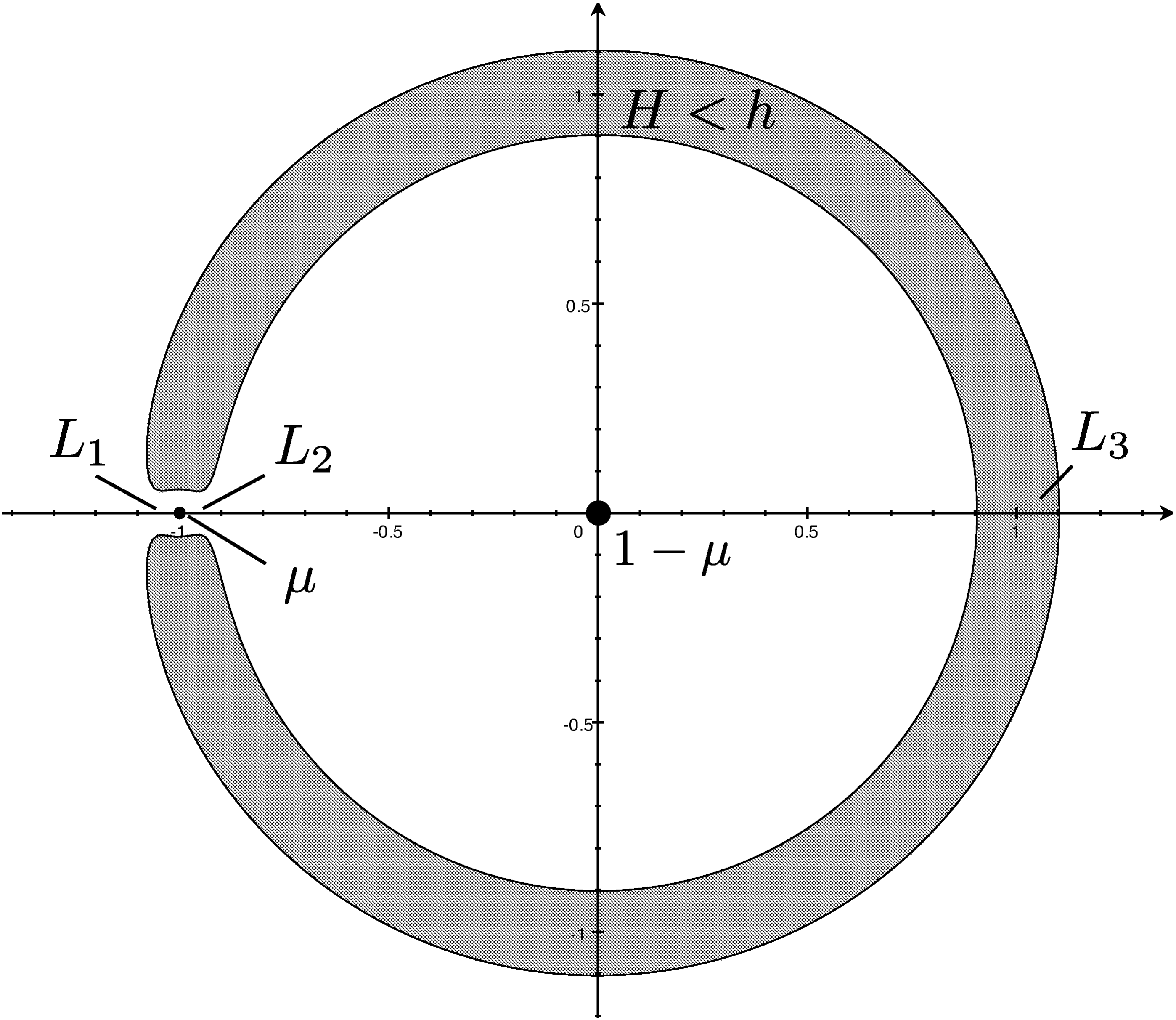}
\end{center}
\caption{The Hill's region for the energy level $h=1.515$ of comet Oterma in
the Jupiter-Sun system.}
\label{fig:forbidden-region}
\end{figure}

The problem has three equilibrium points on the $x$-axis, $L_{1},L_{2},L_{3}$
(see Figure \ref{fig:forbidden-region}), called the Lagrangian points. We
shall be interested in the dynamics associated with $L_{2}$, and with orbits
of energies higher than that of $L_{2}$. The linearized vector field at the
point $L_{2}$ has two real and two purely imaginary eigenvalues, thus by the
Lyapunov theorem (see for example \cite{Simo}, \cite{Moser}) for energies $h$
larger and sufficiently close to $H(L_{2})$ there exists a family of
periodic orbits parameterized by energy emanating from the equilibrium point
$L_{2}.$ Numerical evidence shows that this family extends up to, and even
goes beyond, the smaller primary $\mu$ \cite{Broucke}.

\begin{figure}[tbp]
\begin{center}
\includegraphics[height=2.0in]{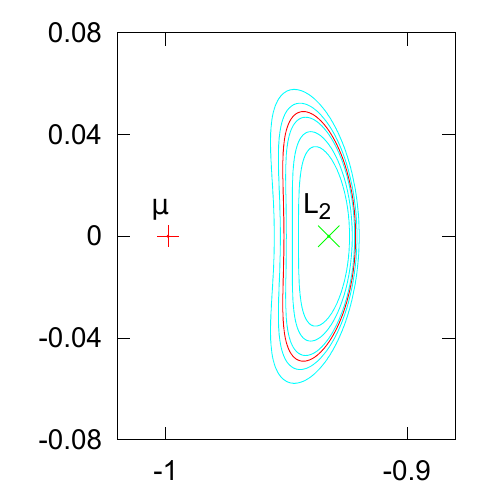}
\end{center}
\caption{A family of Lyapunov orbits in blue, together with the Lyapunov
orbit with energy close to that of the comet Oterma $h=1.\,\allowbreak 515$
in red.}
\label{fig:LapOrbits}
\end{figure}

The PCR3BP admits the following reversing symmetry
\begin{equation*}
\mathbb{S}(x,y,p_{x},p_{y})=(x,-y,-p_{x},p_{y}).
\end{equation*}
For the flow $\Phi_{t}(\mathbf{x})$ of (\ref{eq:PRC3BP}) we have%
\begin{equation}
\mathbb{S}(\Phi_{t}(\mathbf{x}))=\Phi_{-t}(\mathbb{S}(\mathbf{x})).
\label{eq:sym-prop}
\end{equation}
We say that an orbit $\Phi_{t}(\mathbf{x}_{0})$ is $\mathbb{S}$-symmetric
when
\begin{equation}
\mathbb{S}(\Phi_{t}(\mathbf{x}_{0}))=\Phi_{-t}(\mathbf{x}_{0}).
\label{eq:S-sym}
\end{equation}

Each Lyapunov orbit is $\mathbb{S}$-symmetric. When considered on the
constant energy manifold $M(h)$, each Lyapunov orbit is hyperbolic. It
possesses a two dimensional stable manifold and a two dimensional unstable
manifold. These manifolds lie on the same energy level as the orbit, and are
$\mathbb{S}$-symmetric with respect to each other, meaning that the stable
manifold is an image by $\mathbb{S}$ of the unstable manifold (see Figures %
\ref{fig:plot3d}, \ref{fig:plot2d}).

One can choose starting points $q(x^{\ast })$ on the Lyapunov orbits, of the
following form
\begin{equation*}
q(x^{\ast })=(x^{\ast },0,0,\kappa (x^{\ast })),
\end{equation*}%
where $\kappa (x^{\ast })$, representing the $p_y$-coordinate of the point $%
q(x^{\ast })$, is a smooth function that results from the energy condition (%
\ref{eq:M-energy}). Since each Lyapunov orbit intersects the $x$-axis at two
points (see Figure \ref{fig:LapOrbits}), there are two possible choices of $%
x^{\ast }$. One is to the left of $\pi_{x}L_{2},$ the other to the right. We
choose $x^{\ast }$ to be the point on the left. We use the notation $%
L(x^{\ast})$ to denote the Lyapunov orbit which includes $q(x^{\ast })$, and
the notation $T(x^{\ast })$ to denote the period of $L(x^{\ast})$.

We can choose a closed interval $\mathbf{I}\subset \mathbb{R}$ and consider
a normally hyperbolic manifold with a boundary defined as
\begin{equation}
\Lambda =\left\{ \Phi _{s}(q(x^{\ast }))|x^{\ast }\in \mathbf{I},s\in
\lbrack 0,T(x^{\ast }))\right\} .  \label{eq:Lambda-def}
\end{equation}%
Later on, as we formulate our computer assisted results, we shall specify
exactly what interval $\mathbf{I}$ is chosen. We can take a reference
manifold $\mathcal{N}$ for $\Lambda $ of the form,%
\begin{equation}
\mathcal{N}=\left\{ \left( x^{\ast },\theta \right) :x^{\ast }\in \mathbf{I}%
,\theta \in \lbrack 0,2\pi )\right\} ,  \label{eq:N-3bp}
\end{equation}%
with a parameterization%
\begin{equation}
k_{0}(x^{\ast },\theta )=\Phi _{\theta T(x^{\ast })/2\pi }(q(x^{\ast })).
\label{eq:k0-3bp}
\end{equation}%
On $\mathcal{N}$ we consider the flow $r_{0,t}$ induced by $\Phi _{t}$ via $%
k_{0}$,
\begin{equation}
r_{0,t}(x^{\ast },\theta )=\left( x^{\ast },\theta +t2\pi /T(x^{\ast
})\right) ,  \label{eq:r0-3bp}
\end{equation}%
so we naturally have
\begin{equation}
\Phi _{t}\circ k_{0}=k_{0}\circ r_{0,t},  \label{eq:cohomology-3bp}
\end{equation}%
as required.

\subsection{Some numerical observations}

Numerical evidence suggests that $\frac{d}{dx^{\ast }}T(x^{\ast })\neq 0.$
We claim that this implies that $r_{0,t}$ is a twist on $\mathcal{N}$. One
can immediately see that $r_{0,t}$ is exact symplectic. The twist condition
in these coordinates is
\begin{equation*}
\frac{\partial \text{pr}_{\theta }(r_{0,t})}{\partial x^{\ast }}(x^{\ast
},\theta )=-\frac{t2\pi }{T(x^{\ast })^{2}}\frac{dT(x^{\ast })}{dx^{\ast }}%
\neq 0,
\end{equation*}%
where $\text{pr}_{\theta }$ is the projection mapping onto the $\theta $%
-coordinate. This condition is obviously implied by the condition $\frac{d}{%
dx^{\ast }}T(x^{\ast })\neq 0.$

\begin{figure}[tbp]
\begin{center}
\includegraphics[height=3in]{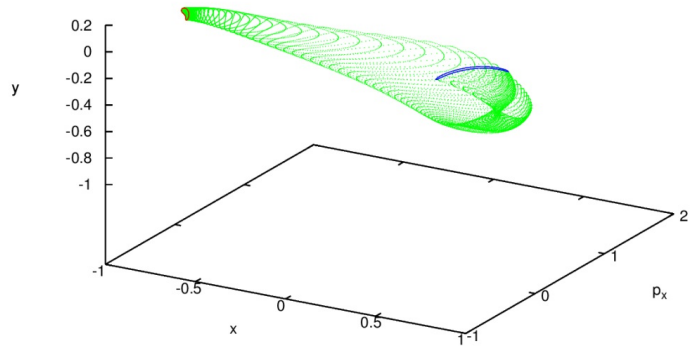}
\end{center}
\caption{A Lyapunov orbit in red, its unstable manifold in green, and the
intersection of the unstable manifold with section $\{y=0\}$ in blue,
projected onto $x,y,p_{x}$ coordinates. The figure is for the energy of
comet Oterma $h=1.515$ in the Jupiter-Sun system.}
\label{fig:plot3d}
\end{figure}

\begin{figure}[tbp]
\begin{center}
\includegraphics[height=1.9in]{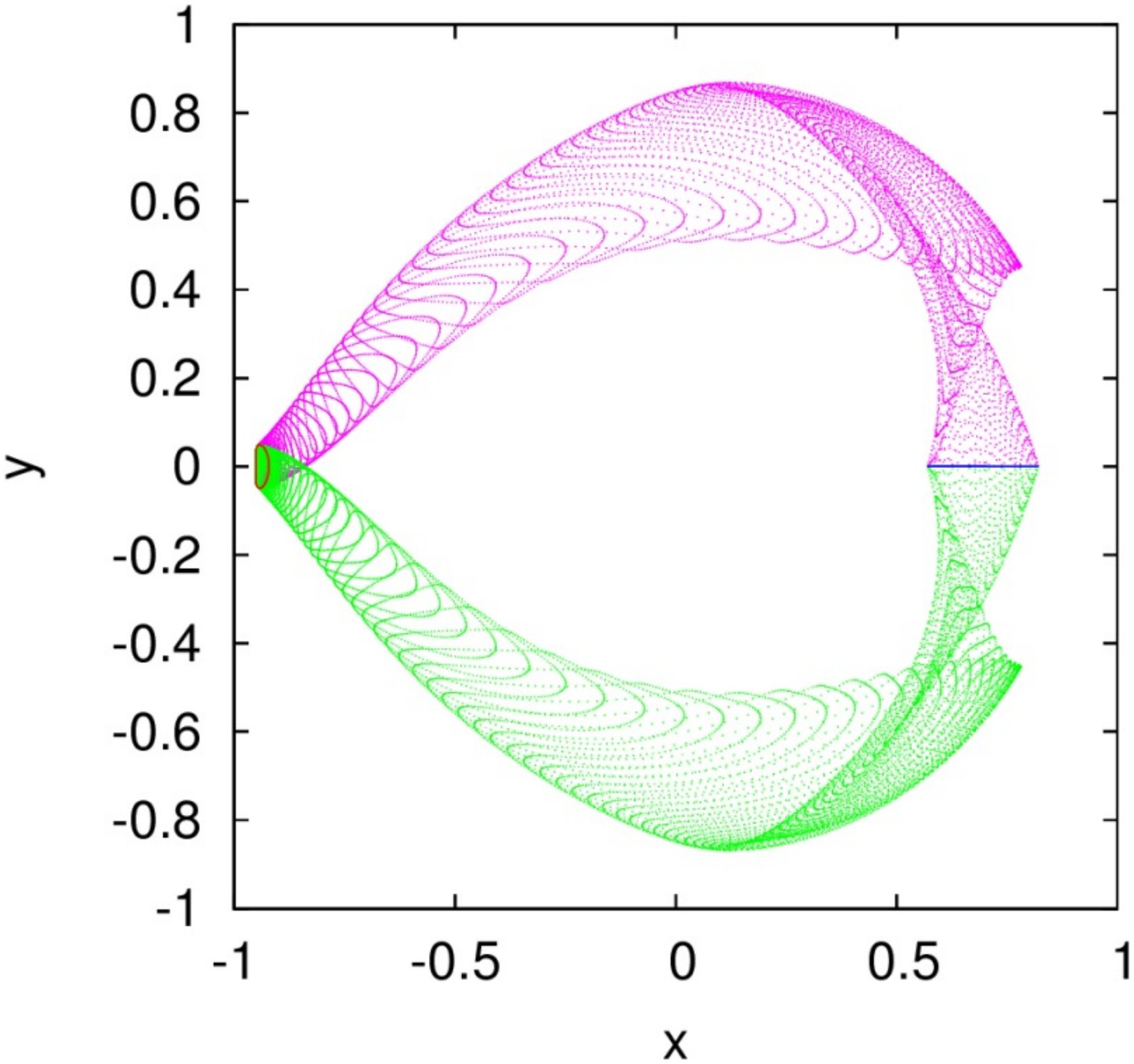}%
\includegraphics[height=1.9in]{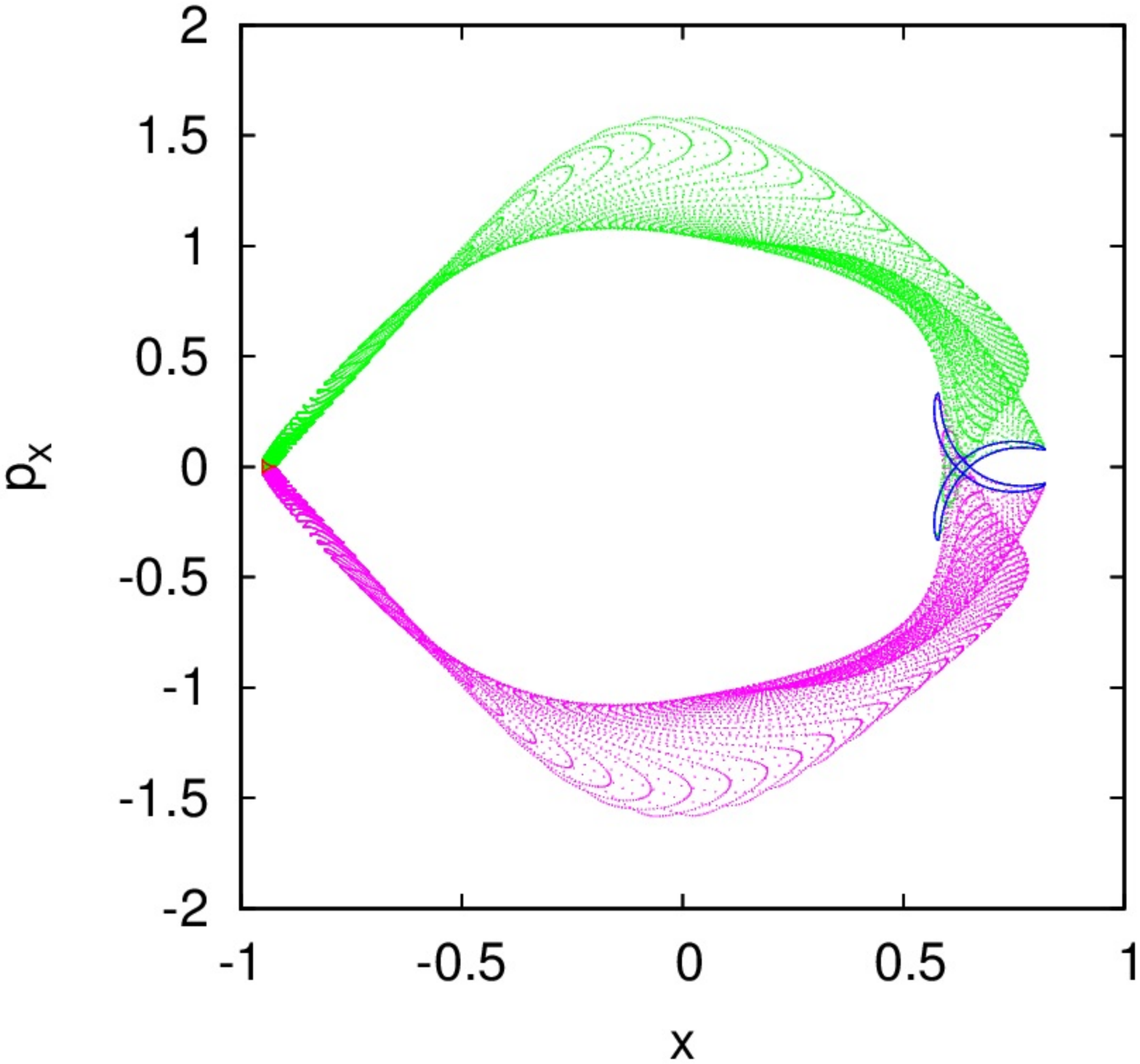}
\end{center}
\caption{The Lyapunov orbit in red, its unstable manifold in green, stable
manifold in purple, and their intersections with section $\{y=0\}$ in blue,
projected onto $x,y$ coordinates (left) and $x,p_{x}$ coordinates (right).
The figure is for the energy of comet Oterma $h=1.\,\allowbreak 515$ in the
Jupiter-Sun system.}
\label{fig:plot2d}
\end{figure}
\begin{figure}[tbp]
\begin{center}
\includegraphics[height=1.9in]{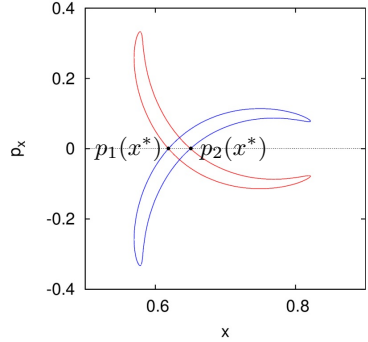}
\end{center}
\caption{The intersection of the manifolds $W^{u}(L(x))$ with $y=0$ in blue
and $W^{s}(L(x))$ with $y=0$ in red.}
\label{fig:plotInter}
\end{figure}

Numerical evidence suggests that intersection of $W^{u}\left( L(x^{\ast
})\right) $ and $W^{s}\left( L(x^{\ast })\right) $ contains four homoclinic
orbits. These orbits pass through four points, which can be seen on Figure %
\ref{fig:plot2d} (right), and Figure \ref{fig:plotInter}. The intersections
of $W^{u}\left( L(x^{\ast })\right) $ and $W^{s}\left( L(x^{\ast })\right) $
with $\{y=0\}$ are the \textquotedblleft banana-shaped\textquotedblright\
loops, which intersect along the four points. Two of these, on the $x$-axis,
are the points from which start the $\mathbb{S}$-symmetric homoclinic
orbits. The remaining two homoclinic orbits are not $\mathbb{S}$-symmetric.
Let us use the notation $p_{1}(x^{\ast })$ for the left and $p_{2}\left(
x^{\ast }\right) $ for the right one of the two $\mathbb{S}$-symmetric
intersection points.

Since the points $p_{1}\left( x^{\ast }\right) ,p_{2}\left( x^{\ast }\right)
$ lie on $W^{u}\left( L(x^{\ast })\right) ,$ for any $x^{\ast }\in \mathbf{I}
$ there exist two numbers $\omega _{1}(x^{\ast })$, $\omega _{2}(x^{\ast })$
such that $p_{i}(x^{\ast })\in W^{u}\left( k_{0}(x^{\ast },\omega
_{i}(x^{\ast }))\right) $. For any $\theta \in \mathbb{R}$,%
\begin{equation}
\begin{array}{l}
\Phi _{\left( \theta -\omega _{i}(x^{\ast })\right) T(x^{\ast })/2\pi
}\left( p_{i}(x^{\ast })\right) \medskip \\
\in W^{u}\left( \Phi _{\left( \theta -\omega _{i}(x^{\ast })\right)
T(x^{\ast })/2\pi }(k_{0}(x^{\ast },\omega _{i}(x^{\ast })))\right) \medskip
\\
=W^{u}\left( \Phi _{\left( \theta -\omega _{i}(x^{\ast })\right) T(x^{\ast
})/2\pi }\Phi _{\omega _{i}(x^{\ast })T(x^{\ast })/2\pi }(q(x^{\ast
}))\right) \medskip \\
=W^{u}\left( k_{0}(x^{\ast },\theta )\right) .%
\end{array}
\label{eq:wave-Wu}
\end{equation}%
Since $\Phi _{t}(p_{i}(x^{\ast }))$ are $\mathbb{S}$-symmetric, $%
p_{i}(x^{\ast })\in W^{s}\left( k_{0}(x^{\ast },-\omega _{i}(x^{\ast
}))\right) $. This implies
\begin{equation}
\begin{array}{l}
\Phi _{\left( \theta -\omega _{i}(x^{\ast })\right) T(x^{\ast })/2\pi
}\left( p_{i}(x^{\ast })\right) \medskip \\
\in W^{s}\left( \Phi _{\left( \theta -\omega _{i}(x^{\ast })\right)
T(x^{\ast })/2\pi }\left( k_{0}(x^{\ast },-\omega _{i}(x^{\ast }))\right)
\right) \medskip \\
=W^{s}\left( \Phi _{\left( \theta -\omega _{i}(x^{\ast })\right) T(x^{\ast
})/2\pi }\Phi _{-\omega _{i}(x^{\ast })T(x^{\ast })/2\pi }(q(x^{\ast
}))\right) \medskip \\
=W^{s}\left( k_{0}(x^{\ast },-2\omega _{i}(x^{\ast }))\right) .%
\end{array}
\label{eq:wave-Ws}
\end{equation}

We now define four homoclinic channels $\Gamma ^{i,j}$ for $i,j=1,2$ as
\begin{eqnarray}
\Gamma ^{i,1} &=&\{\Phi _{\left( \theta -\omega _{i}(x^{\ast })\right)
T(x^{\ast })/2\pi }(p_{i}(x^{\ast }))|x^{\ast }\in \mathbf{I},\theta \in
(-2\pi +\frac{\pi }{8},\frac{\pi }{8})\}  \label{eq:Gamma-i-def} \\
\Gamma ^{i,2} &=&\left\{ \Phi _{\left( \theta -\omega _{i}(x^{\ast })\right)
T(x^{\ast })/2\pi }(p_{i}(x^{\ast }))|x^{\ast }\in \mathbf{I},\theta \in
(0,2\pi )\right\}  \nonumber
\end{eqnarray}%
(These are depicted in Figure \ref{fig:Gamma}.) In total, we consider angles
for the range $\theta \in \left[ -2\pi ,2\pi \right] $, and for each $\theta
$ we have two homoclinic channels. Numerical evidence suggests that $\Gamma
^{i,j}$ lead to four well defined scattering maps%
\begin{equation*}
s_{0}^{i,j}:\mathcal{N}\rightarrow \mathcal{N}.
\end{equation*}%
In fact, we can use any small fragment from the homoclinic orbits as a
homoclinic channel, hence there are infinitely many of such channels. For
our purposes though, restricting to four channels will turn out to be enough
to apply Theorem \ref{th:mechanism-main}.

We shall now show that $s_{0}^{i,j}$ are of the form (\ref%
{eq:s-rotation-form}). From (\ref{eq:wave-Wu}), (\ref{eq:wave-Ws}) we see
that%
\begin{eqnarray}
\Omega _{-}^{\Gamma ^{i,j}}\left( \Phi _{\left( \theta -\omega _{i}(x^{\ast
})\right) T(x^{\ast })/2\pi }\left( p_{i}(x^{\ast })\right) \right) &
=k_{0}\left( x^{\ast },\theta \right) ,  \label{eq:wave-3bp-1} \\
\Omega _{+}^{\Gamma ^{i,j}}\left( \Phi _{\left( \theta -\omega _{i}(x^{\ast
})\right) T(x^{\ast })/2\pi }\left( p_{i}(x^{\ast })\right) \right) &
=k_{0}\left( x^{\ast },\theta -2\omega _{i}(x^{\ast })\right) ,
\label{eq:wave-3bp-2} \\
\sigma ^{\Gamma ^{i,j}}\left( k_{0}\left( x^{\ast },\theta \right) \right) &
=k_{0}\left( x^{\ast },\theta -2\omega _{i}(x^{\ast })\right) ,
\label{eq:scatter-3bp}
\end{eqnarray}%
hence%
\begin{equation}
s_{0}^{i,j}\left( x^{\ast },\theta \right) =k_{0}^{-1}\circ \sigma ^{\Gamma
^{i}}\circ k_{0}\left( x^{\ast },\theta \right) =\left( x^{\ast },\theta
-2\omega _{i}(x^{\ast })\right) .  \label{eq:s0-c3bp}
\end{equation}

\subsection{Planar elliptic restricted three body problem}

The planar restricted elliptic three body problem (PRE3BP) differs from the
PRC3BP by the fact that the two larger bodies move on elliptic orbits of
eccentricities $\varepsilon$ instead of circular orbits. The period of these
orbits is $2\pi.$

After introducing the rotating coordinates (\ref{eq:x,y-coordinates}), the
Hamiltonian of PRE3BP can be rewritten as
\begin{equation}
H_{\varepsilon}(\mathbf{x},t)=H(\mathbf{x})+\varepsilon G(\mathbf{x}%
,t)+O(\varepsilon^{2}),  \label{eq:H-PRE3BP}
\end{equation}
where $H$ is the Hamiltonian of the PRC3BP (\ref{eq:H-PRC3BP}), $G$ is $2\pi$
periodic over $t$ and is given by the formula \cite{CZ}
\begin{equation}
G=\frac{1-\mu}{\left( r_{1}\right) ^{3}}g(\mu,x,y,t)+\frac{\mu}{\left(
r_{2}\right) ^{3}}g(\mu-1,x,y,t),  \label{eq:Gdef}
\end{equation}%
\begin{equation}
g(\alpha,x,y,t)=\alpha(-2y\sin{t}+x\cos{t})-\alpha^{2}\cos{t}.
\label{eq:g-bar}
\end{equation}

We note that in coordinates (\ref{eq:x,y-coordinates}) the two primaries are
no longer stationary. The larger primary rotates on a small elliptic orbit
around the point $(x,y)=(\mu,0)$. Similarly, the smaller primary rotates on
an elliptic orbit around $(x,y)=(-1+\mu,0)$. Note also that $r_1$ and $r_2$
do not measure the distance of the infinitesimal body to the primaries, but
to the points $(\mu,0)$ and $(-1+\mu,0)$, respectively.

We also note that this is a different coordinate system than the one used by
Szebehely in \cite{Szebehely}, where he uses pulsating coordinates which
places the larger bodies in fixed locations.

The movement of an infinitesimal body under the gravitational pull of the
two primaries is given by the non-autonomous equation%
\begin{equation}
\mathbf{\dot{x}}=J\nabla_{\mathbf{x}}H_{\varepsilon}(\mathbf{x},t).
\label{eq:PRE3BP}
\end{equation}

Let $\Sigma _{t=\tau }=\left\{ \left( \mathbf{x},t\right) |t=\tau \right\} $
and $\Phi _{\varepsilon ,\tau ,2\pi }:\Sigma _{t=\tau }\rightarrow \Sigma
_{t=\tau }$ be the map induced by the time $2\pi $ shift along the flow of $%
H_{\varepsilon }$. As mentioned in the previous section, numerical evidence
suggests that, for $\varepsilon =0$, the period $T(x^{\ast })$ of a Lyapunov
orbit $L(x^{\ast })$, satisfies
\begin{equation}
\frac{d}{dx^{\ast }}T(x^{\ast })\neq 0,  \label{eq:twist-cond}
\end{equation}%
which by (\ref{eq:s0-c3bp}) is equivalent to $\Phi _{\varepsilon =0,\tau
,2\pi }$ being a twist map on $\Lambda $. We therefore formulate the
following theorem.

We consider the system (\ref{eq:H-PRE3BP}) in the phase
space $\mathbb{R}^{4}\times \mathbb{T}^{1}$, extended to include the time.
The energy manifolds in the extended space are of the form $\tilde{M}(h)=M(h)\times \mathbb{T}^{1}$, and 
the vector field is of the form $\tilde X=(X,1)$, where $X$ is the Hamiltonian vector field associated to $H$ on $\mathbb{R}^{4}$.
The manifold $\tilde{\Lambda}=\Lambda \times \mathbb{T}^{1}$ is a normally
hyperbolic invariant manifold with boundary for the flow $\tilde{\Phi}_{t}$
of $\tilde X$  in the extended phase space. (That is the flow associated to the PCR3BP,   for $\varepsilon =0$.) 

\begin{theorem}
\label{th:kam-per3bp} Assume \[\frac{d}{dx^{\ast }}T(x^{\ast })\neq 0\] 
and also  
\begin{equation}
\frac{d}{dx^{\ast }}H(q(x^{\ast }))\neq 0\qquad \text{for }x^{\ast }\in
\mathbf{I}.  \label{eq:dHq-ne-zero-1}
\end{equation}%
Then for sufficiently small perturbation $\varepsilon $ the manifold $\tilde{%
\Lambda}$ is perturbed into a $O(\varepsilon )$ close normally hyperbolic
manifold $\tilde{\Lambda}_{\varepsilon }$, with boundary, which is invariant
under the flow induced by (\ref{eq:PRE3BP}). Moreover, there exists a Cantor
set $\mathfrak{C}_{\varepsilon }$ of invariant tori in $\tilde{\Lambda}%
_{\varepsilon }$.
\end{theorem}

\begin{proof}
The proof is based on combining the persistence result for normally
hyperbolic invariant manifolds theorem (Theorem \ref{th:nhim-pert}), with
the KAM Theorem \ref{th:KAM}. The technical issue that we need to address in
this proof is the fact that $\tilde{\Lambda}$ is a normally hyperbolic
manifold with boundary, but statement of Theorem \ref{th:nhim-pert} is for
compact manifolds without boundary. To apply Theorem \ref{th:nhim-pert} we
will modify the vector field induced by $H_{\varepsilon }$, so that $\tilde{%
\Lambda}$ will become a compact manifold without boundary after the
modification.

We now discuss the modification. Consider two closed intervals $\mathbf{I}%
^{\prime \prime },\mathbf{I}^{\prime }\subset \mathbb{R}$ satisfying $%
\mathbf{I}^{\prime \prime }\subset \mathrm{int}\mathbf{I}^{\prime },$ $%
\mathbf{I}^{\prime }\subset \mathrm{int}\mathbf{I}$ and let
\begin{eqnarray*}
\lbrack a,b] &:&=\{H(L\left( x^{\ast }\right) ):x^{\ast }\in \mathbf{I}\}, \\
\lbrack a^{\prime },b^{\prime }] &:&=\{H(L\left( x^{\ast }\right) ):x^{\ast
}\in \mathbf{I}^{\prime }\}, \\
\lbrack a^{\prime \prime },b^{\prime \prime }] &:&=\{H(L\left( x^{\ast
}\right) ):x^{\ast }\in \mathbf{I}^{\prime \prime }\}.
\end{eqnarray*}%
We have $[a^{\prime \prime },b^{\prime \prime }]\subset (a^{\prime
},b^{\prime }),$ $[a^{\prime },b^{\prime }]\subset (a,b)$. Consider the
following modified Hamiltonian%
\begin{equation}
\widehat{H}_{\varepsilon }(\mathbf{x},t)=H\left( \mathbf{x}\right) +b\left(
H(\mathbf{x})\right) \left[ H_{\varepsilon }(\mathbf{x},t)-H\left( \mathbf{x}%
\right) \right] ,  \label{eq:Mod-1-ham}
\end{equation}%
where $b:\mathbb{R}\rightarrow \lbrack 0,1]$ is a smooth `bump' function,
satisfying $b|_{\left[ a^{\prime \prime },b^{\prime \prime }\right] }=1,$ $%
b|_{\mathbb{R}\setminus \left( a^{\prime },b^{\prime }\right) }=0$.
\begin{figure}[tbp]
\begin{center}
\includegraphics[width=8.2cm]{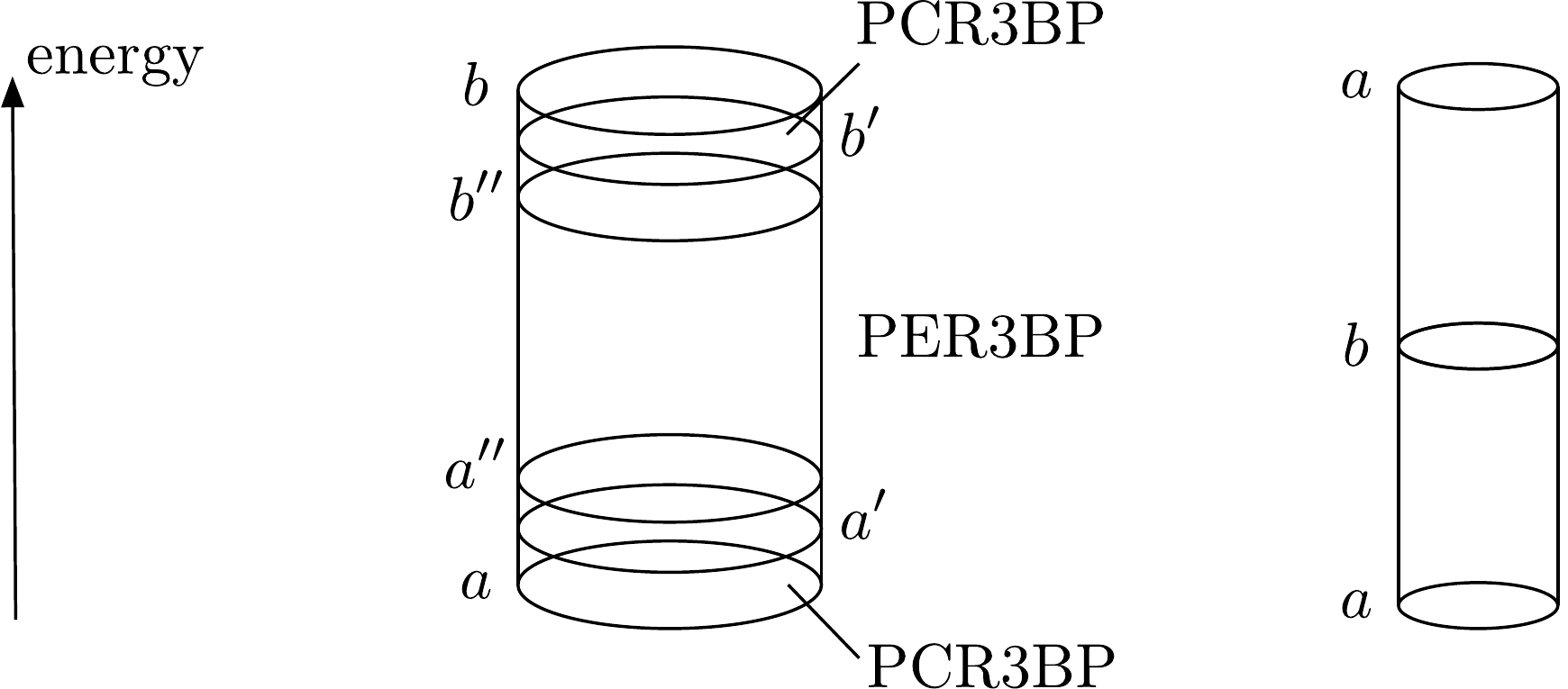}
\end{center}
\caption{Schematic plot of the modified Hamiltonian dynamics (left), and the
gluing (right).}
\label{fig:mod}
\end{figure}

We now note some facts about $\widehat{H}_{\varepsilon }$. For all points $%
\mathbf{x}$ with energies $H(\mathbf{x})$ outside of $\left( a^{\prime
},b^{\prime }\right) $, $\widehat{H}_{\varepsilon }(\mathbf{x},t)=H\left(
\mathbf{x}\right) $. This means that if we start with energies outside of $%
\left( a^{\prime },b^{\prime }\right) $, we are dealing with the PCR3BP (see Figure \ref{fig:mod}),
where energy is preserved. This in particular
implies that if we start with a point with energy in $\left( a^{\prime
},b^{\prime }\right) $, then the energy along the solution of the flow for $%
\widehat{H}_{\varepsilon }$ will never leave $\left( a^{\prime },b^{\prime
}\right) $. (If it were to reach the boundary of $\left( a^{\prime
},b^{\prime }\right) $, this would contradict the fact that on the boundary
of $\left( a^{\prime },b^{\prime }\right) $ the Hamiltonian is autonomous
and the energy is preserved). Another important fact is that for energies
inside of $\left[ a^{\prime \prime },b^{\prime \prime }\right] $ the flow of
(\ref{eq:Mod-1-ham}) coincides with the solution for PER3BP. For $%
\varepsilon =0$, the manifold $\tilde{\Lambda}$ is invariant under the flow
of $\widehat{H}_{\varepsilon }$.

The manifold $\tilde{\Lambda}$ is diffeomorphic to a $3$-dimensional annulus
$\mathbf{I}\times \mathbb{T}^{2}$, with one boundary component $L(a)\times
\mathbb{T}^{1}$ and the other boundary component $L(b)\times \mathbb{T}^{1}$, 
where $L(a)$, $L(b)$ are the Lyapunov orbits at energies $a$, $b$,
respectively. We restrict to the energy level sets $\bigcup_{h\in[a,b]}\tilde{M}(h)$, 
and glue another copy of $\bigcup_{h\in [a,b]}\tilde{M}(h)$, such
that the energy level of $b$ in the first copy matches with the energy level
of $b$ in the second copy (see Figure \ref{fig:mod}). Note that for energies
in $\left[ a,b\right] \setminus \left[ a^{\prime },b^{\prime }\right] $, $\hat H_0$
is preserved along each level set. 
We can modify the vector field so that it
is smooth along the gluing,  that $L(a),L(b)$ remain invariant, and so that the annulus can also be smoothly glued along $L(a)$. This modification is performed outside of $\bigcup_{h\in \lbrack a^{\prime },b^{\prime }]}M(h)$. We denote this vector
field by $\hat X_{\varepsilon }$.
Let $\widehat{M}$ be the manifold in the
extended phase space obtained by the above gluing of level sets, that is $%
\widehat{M}=  \bigcup_{h\in [a,b]}M(h)\sqcup \bigcup_{h\in
\lbrack a,b]}M(h) $. The two copies of $\tilde{\Lambda}$   are smoothly glued along $L(a)\times\mathbb{T}^{1}$, and the corresponding gluing of the 
the two copies of $\tilde{\Lambda}$ inside $\widehat{M}$ results in a $3$%
-dimensional annulus $\widehat{\Lambda }$ whose boundary components at both ends are identical to 
$L(a)\times \mathbb{T}^{1}$. Hence, for $\varepsilon =0$ we have that 
$\widehat{\Lambda }$ is diffeomorphic to a
$3$-dimensional torus. Thus, $\widehat{\Lambda }$ 
can be viewed as a compact, normally hyperbolic invariant manifold  for $\hat X_{0}$. 

For $\varepsilon >0$ sufficiently small, we apply Theorem \ref%
{th:nhim-pert}, to a time shift map along the flow of $\hat{X}_{\varepsilon}$
, obtaining $\widehat{\Lambda }_{\varepsilon }$. 

Note that the flow corresponding to the vector field $\hat X_{\varepsilon}$  
has no effect on the boundary $L(a),L(b)$ of the annulus $\tilde{\Lambda}$.
As $\widehat{\Lambda }$ consists of two copies of $\tilde{\Lambda}$, we can
now restrict to only one of the copies, concluding the existence of a
normally hyperbolic manifold with boundary $\tilde{\Lambda}_{\varepsilon
}\subset \widehat{\Lambda }_{\varepsilon }$ that survives to $\tilde{\Lambda}
$. The boundary of $\tilde{\Lambda}_{\varepsilon }$ is invariant under the
flow of $\hat X_{\varepsilon }$. The only role of $\widehat{\Lambda }$ was to be
able to apply Theorem \ref{th:nhim-pert} \emph{ad litteram}, as $\widehat{%
\Lambda }$ can be viewed as a compact manifold without boundary. From now we
give up entirely on $\widehat{\Lambda }$, and we refer only to $\tilde{%
\Lambda}_{\varepsilon }$. 

We remark here that $\tilde{\Lambda}_{\varepsilon }$ depends on the choice
of the `bump' function $b$. We will address this issue in what follows.

For $\left( \mathbf{x},\tau \right) \in \tilde{\Lambda}$ and fixed $t>0$,
the eigenvalues of $D_{\mathbf{x}}\tilde{\Phi}_{t}(\mathbf{x},\tau)$ are $\lambda
_{1},\lambda _{2},1,1,1,$ with $\left\vert \mathrm{re}\lambda
_{1}\right\vert >1>\left\vert \mathrm{re}\lambda _{2}\right\vert$. We can
therefore choose $\mu $ from Definition \ref{def:nhim} arbitrarily close to
one, which by Theorem \ref{th:nhim-pert} means that $\widehat{\Lambda}%
_{\varepsilon }$ are $C^{l-1}$ smooth with arbitrarily large $l$ (for us it
is enough to have $l\geq 7$).

Let $\Lambda _{\varepsilon ,\tau }=\tilde{\Lambda}_{\varepsilon }\cap \Sigma
_{t=\tau }$. Let $\hat{\Phi}_{\varepsilon ,\tau ,t}$ be the solution of the
ODE induced by (\ref{eq:Mod-1-ham}). Since $\widehat{H}_{\varepsilon }$ is $%
2\pi $ periodic, we can consider the map $\hat{\Phi}_{\varepsilon ,\tau
,2\pi }:\Lambda _{\varepsilon ,\tau }\rightarrow \Lambda _{\varepsilon ,\tau
}$. Our next step is to apply Theorem \ref{th:KAM} to it. It is a well known
property of Hamiltonian systems that a time shift map along a trajectory of
the system is exact symplectic, hence $\hat{\Phi}_{\varepsilon ,\tau ,2\pi }$
is exact symplectic with the standard symplectic form. By Lemma \ref%
{lem:nhim-symplectic}, $\omega |_{\Lambda }$ is non-degenerate. Let us now
consider the reference manifold $\mathcal{N}$ given by (\ref{eq:N-3bp}), and
the parameterization $k_{0}$ from (\ref{eq:k0-3bp}). By Theorem \ref%
{th:k-epsilon}, it is possible to choose coordinates $k_{\varepsilon ,\tau }:%
\mathcal{N}\rightarrow \Lambda _{\varepsilon ,\tau }$ so that $%
k_{\varepsilon ,\tau }^{\ast }\omega |_{\Lambda _{\varepsilon ,\tau
}}=k_{0}^{\ast }\omega |_{\Lambda }=\omega _{\mathcal{N}}$. We can now
define a $C^{l-1}$ smooth family of exact symplectic maps
\begin{equation*}
r_{\varepsilon ,\tau ,2\pi }:\mathcal{N}\rightarrow \mathcal{N}
\end{equation*}%
as
\begin{equation*}
r_{\varepsilon ,\tau ,2\pi }=k_{\varepsilon ,\tau }^{-1}\circ \hat{\Phi}%
_{\varepsilon ,\tau ,2\pi }\circ k_{\varepsilon ,\tau }.
\end{equation*}%
By (\ref{eq:r0-3bp}) and (\ref{eq:twist-cond}) we see that $r_{0,\tau ,2\pi
}=r_{0,2\pi }$ is a twist map. We can therefore apply Theorem \ref{th:KAM}
to obtain a family of invariant tori for $r_{\varepsilon ,\tau ,2\pi }$. The
Lyapunov orbits on $\Lambda $ play the role of the unperturbed tori of
Theorem \ref{th:KAM}. Their rotation numbers are determined by the choice of
$x^{\ast }\in \mathbf{I}.$ We thus have a family of invariant tori $%
u_{x^{\ast }}$ of $r_{\varepsilon ,\tau ,2\pi }$ for a Cantor set $\mathfrak{%
G}_{\varepsilon ,\tau }^{\mathbf{I}}$ of $x^{\ast }$ in $\mathbf{I}.$ We
have now obtained our Cantor set $\mathfrak{C}_{\varepsilon }$ defined by
\begin{equation*}
\mathfrak{C}_{\varepsilon }=\left\{ \left( k_{\varepsilon ,\tau }\left(
u_{x^{\ast }}\right) ,\tau \right) :x^{\ast }\in \mathfrak{G}_{\varepsilon
,\tau }^{\mathbf{I}}\subset \mathbf{I},\tau \in \mathbb{T}^{1}\right\} ,
\end{equation*}%
as claimed in the statement of the theorem. What is left though, is to deal
with the fact that these tori were obtained for the modified Hamiltonian.

Let $\mathbf{I}^{\prime \prime \prime }\subset \mathrm{int}\mathbf{I}%
^{\prime \prime }$ and $\left[ a^{\prime \prime \prime },b^{\prime \prime
\prime }\right] =\{H(L\left( x^{\ast }\right) ):x^{\ast }\in \mathbf{I}%
^{\prime \prime \prime }\}$. By the smooth dependence of the flow on the
parameter $\varepsilon $, for sufficiently small $\varepsilon _{0}$, for any
$\mathbf{x}\in \Lambda _{\varepsilon ,\tau }$ satisfying $H\left( \mathbf{x}%
\right) \in \left[ a^{\prime \prime \prime },b^{\prime \prime \prime }\right]
$, any $t\in \left[ 0,2\pi \right] $ and any $\varepsilon \in \left[
0,\varepsilon _{0}\right] $, we have $H(\hat{\Phi}_{\varepsilon ,\tau
,t}\left( \mathbf{x}\right) )\in \left[ a^{\prime \prime },b^{\prime \prime }%
\right] $. Since for energies from $\left[ a^{\prime \prime },b^{\prime
\prime }\right] $ the trajectory for (\ref{eq:Mod-1-ham}) coinsides with the
trajectroy for the PER3BP, we have that $\hat{\Phi}_{\varepsilon ,\tau
,t}\left( \mathbf{x}\right) =\Phi _{\varepsilon ,\tau ,2\pi }\left( \mathbf{x%
}\right) $, so, the established invariant tori from $\mathfrak{C}%
_{\varepsilon ,\tau }^{\mathbf{I}^{\prime \prime \prime }}$ are in fact true
invariant tori for the PER3BP. We can restrict $\mathfrak{C}_{\varepsilon }$
to these tori, which finishes our proof.
\end{proof}

\begin{remark}
The invariant tori in $\mathfrak{C}_{\varepsilon }$ separate $\tilde{\Lambda}%
_{\varepsilon }$ forming an obstruction, which prohibits a large change of
energy using only the inner dynamics on $\tilde{\Lambda}_{\varepsilon }$.
\end{remark}

\begin{remark}
For existence of diffusion we need Theorem \ref{th:kam-per3bp} only for the
persistence of the manifold $\tilde{\Lambda}$ under perturbation. Since $%
\tilde{\Lambda}$ is a manifold with a boundary, the KAM part of the proof of
Theorem \ref{th:kam-per3bp} allows us to obtain persistence of the
boundaries of the manifold under perturbation.
\end{remark}

We now turn to the computation of perturbed scattering maps.

\begin{lemma}
\label{lem:scatter-e3bp}For any $x^{\ast }\in \mathbf{I}$ and $i,j=1,2$%
\begin{equation}
s_{\varepsilon ,\tau }^{i,j}=s_{0}^{i,j}+\varepsilon J\nabla S_{0,\tau
}^{i,j}\circ s_{0}^{i,j}+O\left( \varepsilon ^{2}\right) ,
\label{eq:s-for-3bp}
\end{equation}%
and%
\begin{eqnarray*}
\frac{\partial S_{0,\tau }^{i,j}}{\partial \theta }(s_{0}^{i,j}\left(
x^{\ast },\theta \right) )& =\frac{T(x^{\ast })}{2\pi }[-G\left( \Phi
_{\theta T(x^{\ast })/2\pi }\left( q(x^{\ast })\right) ,\tau \right) \\
& \qquad \quad +G\left( \Phi _{\left( \theta -2\omega _{i}(x^{\ast })\right)
T(x^{\ast })/2\pi }\left( q(x^{\ast })\right) ,\tau \right) \\
& \quad -\int_{-\infty }^{0}\frac{\partial G}{\partial t}\left( \Phi
_{u+(\theta -\omega _{i}(x^{\ast }))T(x^{\ast })/2\pi }\left( p_{i}(x^{\ast
})\right) ,\tau +u\right) \\
& \qquad \quad -\frac{\partial G}{\partial t}\left( \Phi _{u+\theta
T(x^{\ast })/2\pi }\left( q(x^{\ast })\right) ,\tau +u\right) du \\
& \quad -\int_{0}^{\infty }\frac{\partial G}{\partial t}\left( \Phi
_{u+\left( \theta -\omega _{i}(x^{\ast })\right) T(x^{\ast })/2\pi }\left(
p_{i}(x^{\ast })\right) ,\tau +u\right) \\
& \qquad \quad -\frac{\partial G}{\partial t}\left( \Phi _{u+\left( \theta
-2\omega _{i}(x^{\ast })\right) T(x^{\ast })/2\pi }\left( q(x^{\ast
})\right) ,\tau +u\right) du].
\end{eqnarray*}
\end{lemma}

\begin{proof}
The proof follows by substituting (\ref{eq:wave-3bp-1}--\ref{eq:s0-c3bp})
into (\ref{eq:s-eps-nonaut}--\ref{eq:S0tau-def}) and computing the partial
derivative with respect to $\theta $. We give the details in \ref%
{sec:proof-scatter-lem}.
\end{proof}

We are now ready to formulate a theorem that can be used to obtain a proof
of diffusion in the restricted three body problem.

\begin{theorem}
\label{th:diffusion-3bp}Assume that for a given interval $\mathbf{I}$, for
any $x^{\ast }\in \mathbf{I}$ the manifolds $W^{u}\left( L(x^{\ast })\right)
,$ $W^{s}\left( L(x^{\ast })\right) $ intersect transversally, and that we
have two points $p_{1}(x^{\ast })\neq p_{2}(x^{\ast })$
\begin{equation*}
p_{1}(x^{\ast }),p_{2}(x^{\ast })\in W^{u}\left( L(x^{\ast })\right) \cap
W^{s}\left( L(x^{\ast })\right) \cap \{y=0,p_{x}=0\}.
\end{equation*}%
Assume that for $i,j=1,2$ the $\Gamma ^{i,j}$ defined in (\ref%
{eq:Gamma-i-def}) are homoclinic channels. Assume also that there exists $%
\tau ^{\ast }\in (0,2\pi ]$ such that for any $x^{\ast }\in \mathbf{I}$ and
for any $\theta \in \mathbb{T}^{1}$ there exist $i_{1},i_{2},j_{1},j_{2}\in
\{1,2\},$ such that%
\begin{eqnarray}
\frac{\partial S_{0,\tau ^{\ast }}^{i_{1},j_{1}}}{\partial \theta }%
(s_{0}^{i_{1},j_{1}}\left( x^{\ast },\theta \right) )& >0,
\label{eq:scatter-grad-1} \\
\frac{\partial S_{0,\tau ^{*}}^{i_{2},j_{2}}}{\partial \theta }%
(s_{0}^{i_{2},j_{2}}\left( x^{\ast },\theta \right) )& <0.
\label{eq:scatter-grad-2}
\end{eqnarray}%
Also assume that the assumptions of Theorem \ref{th:kam-per3bp} hold. Then
there exists an $\varepsilon ^{\ast }>0$ such that for all $\varepsilon \in
(0,\varepsilon ^{\ast })$, any $I_{1},I_{2}\in \mathbf{I}$, $I_{1}<I_{2}$
and any $\rho >0$ there exist heteroclinic orbits from $B_{\rho }^{<}(I_{1})$
to $B_{\rho }^{>}(I_{2})$ and a heteroclinic orbit from $B_{\rho
}^{>}(I_{2}) $ to $B_{\rho }^{<}(I_{1})$.
\end{theorem}

\begin{proof}
By Theorem \ref{th:kam-per3bp} the manifold persists under perturbation.

By (\ref{eq:scatter-grad-1}) the assumption (\ref{eq:scatter-grad-a1}) from
Theorem \ref{th:mechanism-main} holds, hence there exists a homoclinic orbit
from $B_{\rho }^{<}(I_{1})$ to $B_{\rho }^{>}(I_{2})$.

The proof of a homoclinic orbit from $B_{\rho }^{>}(I_{2})$ to $B_{\rho
}^{<}(I_{1})$ also follows from Theorem \ref{th:mechanism-main}, using (\ref%
{eq:scatter-grad-2}).
\end{proof}

\begin{remark}
For any $I_{1}<I_{2},$ taking sufficiently small $\rho $ ensures that the
set $B_{\rho }^{>}(I_{2})$ has higher energy than $B_{\rho }^{<}(I_{1})$.
Thus the homoclinic orbits from Theorem \ref{th:diffusion-3bp} involve a
change of energy for arbitrarily small $\varepsilon >0.$ The size of the
change of the energy is determined by the choice of $I_{1},I_{2}$ and does
not depend on $\varepsilon$.
\end{remark}

\begin{remark}
Note that Theorem \ref{th:diffusion-3bp} establishes existence of diffusion
by showing that if there is no inner diffusion, there is diffusion along
homoclinic excursions. So, in either case there is diffusion. The KAM
theorem shows that there is no inner diffusion hence it is the alternative
of homoclinic excursions that happens.
\end{remark}

\section{Numerical verification of the hypothesis of Theorem \protect\ref%
{th:diffusion-3bp}}

\label{sec:num}

We shall give numerical evidence that assumptions (\ref{eq:scatter-grad-1}--%
\ref{eq:scatter-grad-2}) from Theorem \ref{th:diffusion-3bp} are satisfied.
For this we need to numerically compute integrals along the orbits $\Phi
_{t}\left( p_{1}(x^{\ast })\right) $, $\Phi _{t}\left( p_{2}(x^{\ast
})\right) $ and $\Phi _{t}\left( q(x^{\ast })\right) .$ We describe the
procedure of how this can be done.

The $q(x^{\ast })$ and $T(x^{\ast })$ are found as follows. We consider a
section $\Sigma =\left\{ y=0\right\} $ and a Poincar\'{e} map $P:\Sigma
\rightarrow \Sigma $. If for a point $q=(x,0,0,p_{y})\in \Sigma $ we have $%
\pi _{p_{x}}P(q)=0$, then by the symmetry property (\ref{eq:sym-prop}), the
point $q$ lies on a periodic orbit (the Poincar\'{e} map $P$ makes a half
turn along the orbit starting from $q$). Thus, for a given $x^{\ast }$, the
point $q(x^{\ast })=(x,0,0,\kappa (x^{\ast }))$ can be found by considering
a function $h:\mathbb{R\rightarrow R}$ defined as $h(p_{y})=\pi
_{p_{x}}P(x^{\ast },0,0,p_{y}),$ and numerically solving $h(p_{y})=0$.
Having found $q(x^{\ast }),$ the first time to reach $\Sigma $ along the
flow is $T(x^{\ast })/2.$

We now discuss how we compute $p_{1}(x^{\ast })$ and $p_{2}(x^{\ast }).$ The
$q(x^{\ast })$ is a fixed point for the $T(x^{\ast })$ time shift along the
trajectory map $\Phi _{T(x^{\ast })}$. The one dimensional unstable manifold
of the map $\Phi _{T(x^{\ast })}$ at $q(x^{\ast })$ is equal to the unstable
fiber $W^{u}(q(x^{\ast }))$ of the flow. We consider the matrix $A=D\Phi
_{T(x^{\ast })}(q(x^{\ast })).$ The unstable eigenvector $v$ of $A$ is
colinear with $T_{q(x^{\ast })}W^{u}(q(x^{\ast })).$ For sufficiently small $%
h\in \mathbb{R}$, the point $q_{h}=q(x^{\ast })+hv$ is a good approximation
of a point on $W^{u}(q(x^{\ast })),$ and $v$ is a good approximation of $%
T_{q_{h}}W^{u}(q(x^{\ast }))$. We can propagate $q_{h}$ along the flow to
the section $\Sigma _{x>0}=\{y=0,x>0\}$. Let $\tau (h)$ be the time from $%
q_{h}$ to $\Sigma _{x>0}$. If we can find such $h$, so that $\pi
_{p_{x}}\Phi _{\tau (h)}(q_{h})=0$, then by the symmetry property (\ref%
{eq:sym-prop}), the point $\Phi _{\tau (h)}(q_{h})$ is an approximation of a
point on a symmetric homoclinic orbit to $L(x^{\ast })$. This way we
numerically compute $p_{i}(x^{\ast })$ as $\Phi _{\tau (h_{i})}(q_{h_{i}})$,
for $i=1,2$. From the computation we also obtain $\omega _{i}(x^{\ast
})=2\pi \tau (h_{i})/T(x^{\ast })$\textrm{\ mod }$2\pi $, for $i=1,2$.

For any $q\in L(x^{\ast }),$ the unstable fiber $W^{u}(q)$ can be computed
by propagating $W^{u}(q(x^{\ast }))$ along the flow. Thus, since $%
p_{i}(x^{\ast })_{-}\in L(x^{\ast })$, the $T_{p_{i}(x^{\ast
})}W^{u}(p_{i}(x^{\ast })_{-})$ are approximated by $D\Phi _{\tau
(h_{i})}(q_{h_{i}})v$, for $i=1,2$.

We now focus on $x^{\ast }=-0.95$. Using the above outlined procedure, we
obtain the following numerical results:%
\begin{eqnarray*}
q(x^{\ast }) &=&\left( -0.95,0,0,-0.8413472441\right) , \\
T(x^{\ast }) &=&3.041751775, \\
p_{1}(x^{\ast }) &=&\left( 0.6207553555,0,0,1.38203433\right) , \\
p_{2}(x^{\ast }) &=&\left( 0.6514581118,0,0,1.334413389\right) , \\
\omega _{1}(x^{\ast }) &=&1.451540621, \\
\omega _{2}(x^{\ast }) &=&-0.2527863329,
\end{eqnarray*}%
and $T_{p_{i}(x^{\ast })}W^{u}(p_{i}(x^{\ast })_{-})=\mathrm{span}\left\{
v_{i}\right\} $, for $i=1,2$, with%
\begin{equation}
\begin{array}{rrrr}
v_{1}~=~(1, & -9.823658901, & 17.9416819, & -1.60121149),\medskip \\
v_{2}~=~(1, & -4.42688411, & 8.405095683, & -1.503579624).%
\end{array}
\label{eq:vi-num}
\end{equation}

We use a linearization to approximate $W^{u}(q).$ Close to $L(x^{\ast })$,
such approximation is accurate. We integrate the flow using a high precision
(order 20) Taylor method. Thus, the resulting approximation of the points $%
p_{1}(x^{\ast }),$ $p_{2}(x^{\ast })$ is reliable.

Based on the above $p_{1}(x^{\ast }),$ $p_{2}(x^{\ast }),$ we can
numerically compute the homoclinic orbits $\Phi _{t}\left( p_{1}(x^{\ast
})\right) $, $\Phi _{t}\left( p_{2}(x^{\ast })\right) $ (see Figure \ref%
{fig:Gamma095}), together with the corresponding fragments of the homoclinic
channels $\Gamma ^{i,j}$ defined in (\ref{eq:Gamma-i-def}), for $i,j=1,2$.

\begin{figure}[tbp]
\begin{center}
\begin{tabular}{l}
\includegraphics[height=1.9in]{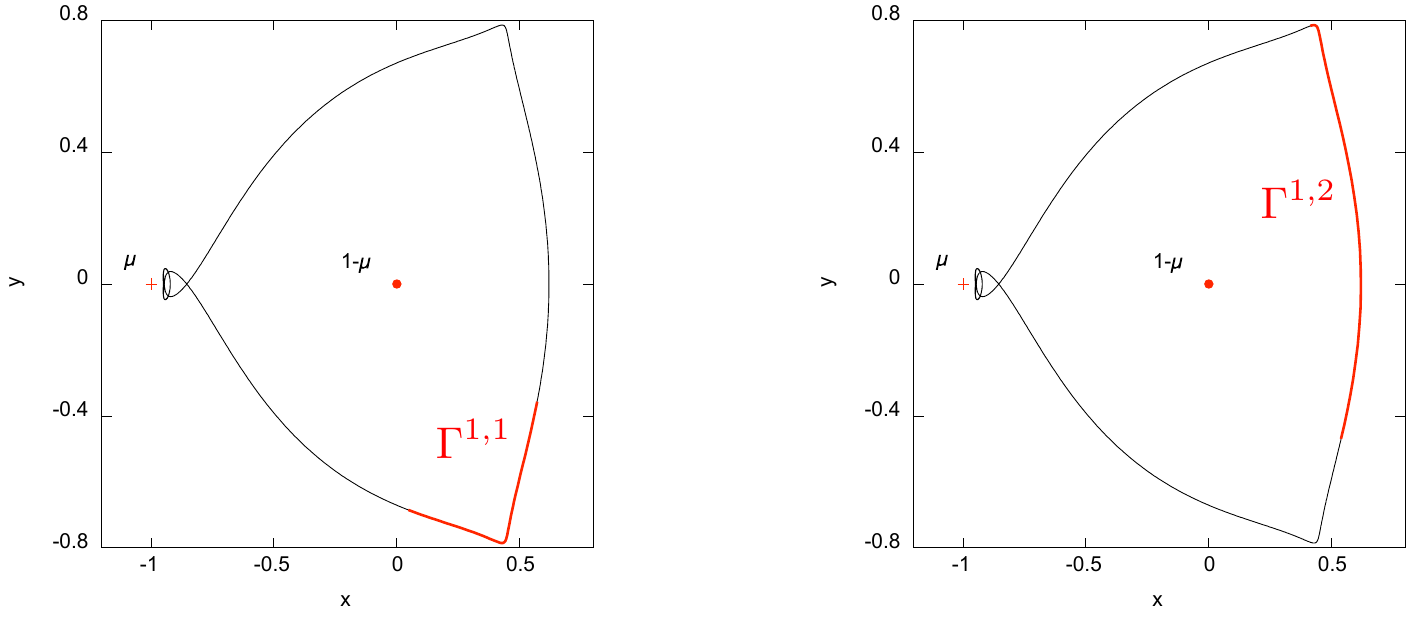} \\
\includegraphics[height=1.9in]{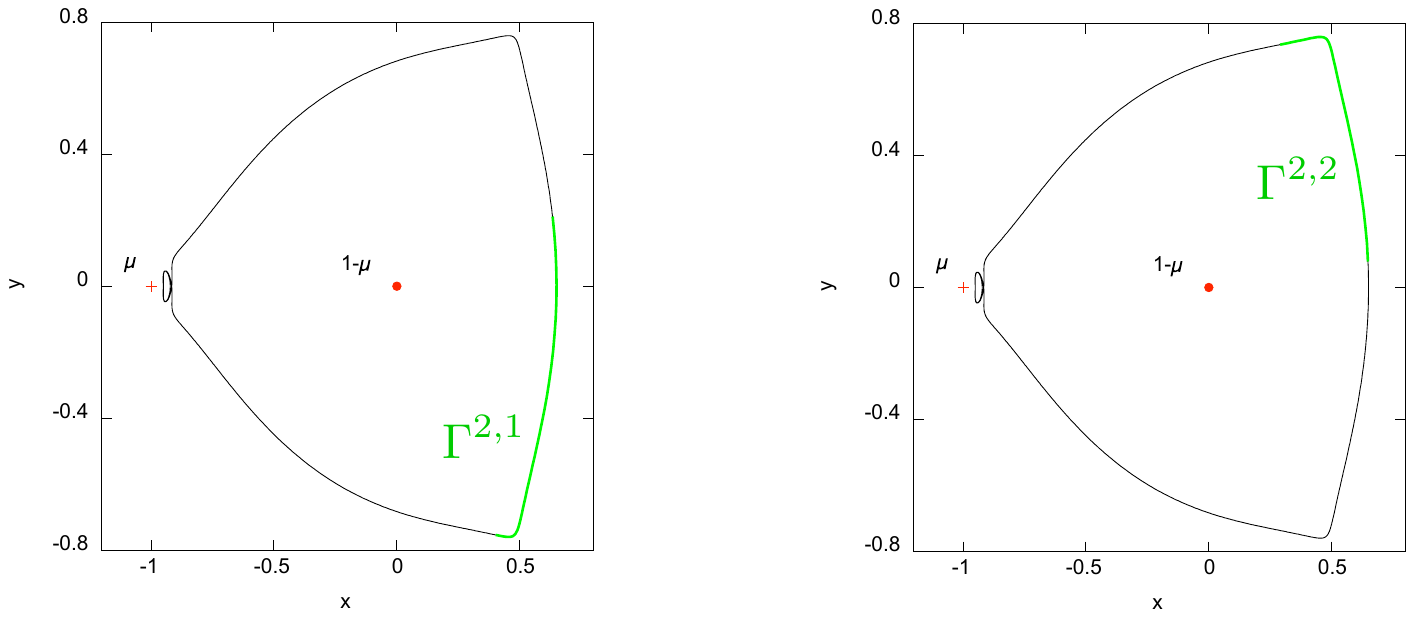}%
\end{tabular}%
\end{center}
\caption{Homoclinic orbit $\Phi _{t}(p_{1}(x^{\ast }))$ (top) and $\Phi
_{t}(p_{2}(x^{\ast }))$ (bottom) in black, and their fragments on $\Gamma
^{i,j}$, in green and red, for $x^{\ast }=-0.95.$}
\label{fig:Gamma095}
\end{figure}
\begin{figure}[tbp]
\begin{center}
\includegraphics[height=1.55in]{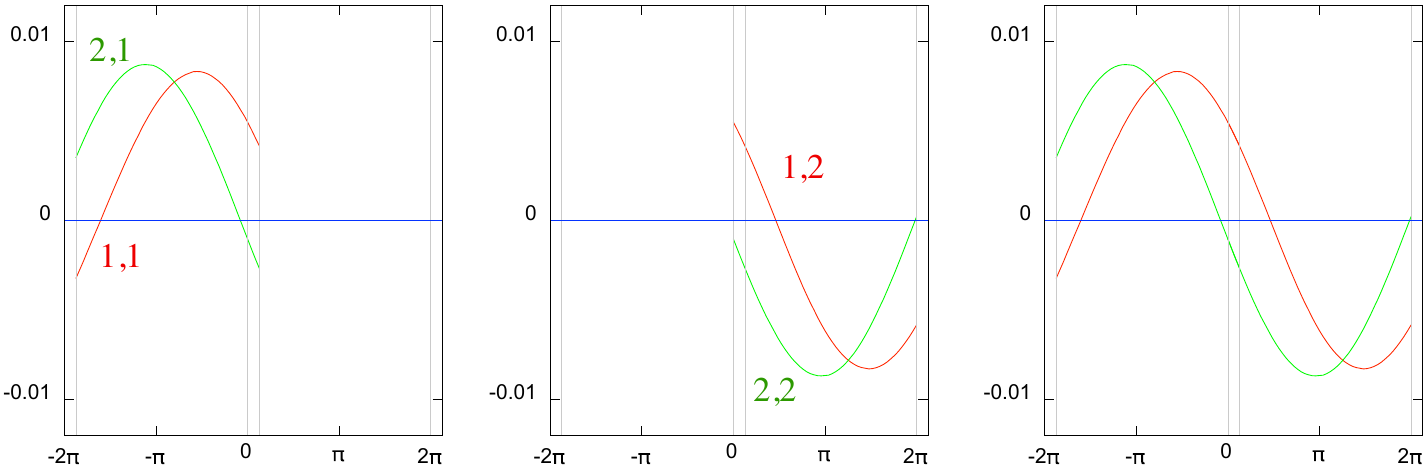}
\end{center}
\caption{The plots of $\protect\theta \rightarrow \frac{\partial S_{0,%
\protect\tau ^{\ast }}^{i,j}}{\partial \protect\theta }(s_{0}^{i,j}(x^{\ast
},\protect\theta )),$ for $\protect\tau ^{\ast }=0$ and $x^{\ast }=-0.95.$
The corresponding indexes $i,j$ are indicated on the plot. }
\label{fig:meln2dx}
\end{figure}

The $\Gamma ^{i,j}$ are indeed homoclinic channels, i.e. satisfy conditions
(i) and (ii) from Definition \ref{def:homoclinic-channel}. We describe how
this has been investigated. Condition (i) follows from the fact that $%
W^{s}\left( \Lambda \right) $ and $W^{u}\left( \Lambda \right) $ intersect
transversally. To show (ii) we shall first fix an $i\in \{1,2\}$ and
consider $\mathbf{x}=p_{i}(x^{\ast })\in \Gamma ^{i,1}\cup \Gamma ^{i,2}.$
We need to verify that%
\begin{equation}
T_{\mathbf{x}}\Gamma ^{i}\oplus T_{\mathbf{x}}W^{u}\left( \mathbf{x}%
_{-}\right) =T_{\mathbf{x}}W^{u}\left( \Lambda \right) .
\label{eq:Hom-chan-cond-2}
\end{equation}%
We shall use the fact that%
\begin{equation*}
T_{\mathbf{x}}\Gamma ^{i}=\mathrm{span}\left( \frac{d}{dx^{\ast }}%
p_{i}(x^{\ast }),F(p_{i}(x^{\ast }))\right) .
\end{equation*}%
Focusing on $x^{\ast }=-0.95,$ the numerical results are:%
\begin{equation}
\begin{array}{rrrr}
\mathrm{span}\left\{ \frac{d}{dx^{\ast }}p_{1}(x^{\ast })\right\} =\,\mathrm{%
span}\{(1, & 0, & 0, & -1.826312946)\},\medskip \\
\mathrm{span}\left\{ F(p_{1}(x^{\ast }))\right\} =\,\mathrm{span}\{(0, & 1,
& -1.60121149, & 0)\},%
\end{array}
\label{eq:Tgammai-num}
\end{equation}%
and%
\begin{equation}
\begin{array}{rrrr}
\mathrm{span}\left\{ \frac{d}{dx^{\ast }}p_{2}(x^{\ast })\right\} =\,\mathrm{%
span}\{(1, & 0, & 0, & -1.898743306)\},\medskip \\
\mathrm{span}\left\{ F(p_{2}(x^{\ast }))\right\} =\,\mathrm{span}\{(0, & 1,
& -1.503579624, & 0)\}.%
\end{array}
\label{eq:Tgammai-num-2}
\end{equation}%
Comparing (\ref{eq:Tgammai-num}) with (\ref{eq:vi-num}), it can easily be
checked that $v_{1},$ $\frac{d}{dx^{\ast }}p_{1}(x^{\ast }),$ $%
F(p_{1}(x^{\ast }))$ are linearly independent. From (\ref{eq:Tgammai-num-2})
and (\ref{eq:vi-num}) we also see that so are $v_{2},$ $\frac{d}{dx^{\ast }}%
p_{2}(x^{\ast }),$ $F(p_{2}(x^{\ast }))$. This implies (\ref%
{eq:Hom-chan-cond-2}). Transversality is preserved along the flow, which
demonstrates that (ii) holds for the unstable manifold. For the stable
manifold, (ii) follows from the symmetry of the PCR3BP.

We now move to the computation of $\frac{\partial S_{0,\tau ^{\ast }}^{i,j}}{%
\partial \theta }$. We take $\tau ^{\ast }=0.$ Employing Lemma \ref%
{lem:scatter-e3bp}, we can numerically compute and plot the functions
\begin{equation*}
\theta \rightarrow \frac{\partial S_{0,\tau ^{\ast }}^{i,j}}{\partial \theta
}(s_{0}^{i,j}(x^{\ast },\theta )),
\end{equation*}%
for $i,j=1,2$. The plots are given in Figure \ref{fig:meln2dx}.
\begin{figure}[tbp]
\begin{center}
\begin{tabular}{l}
\includegraphics[height=2.0in]{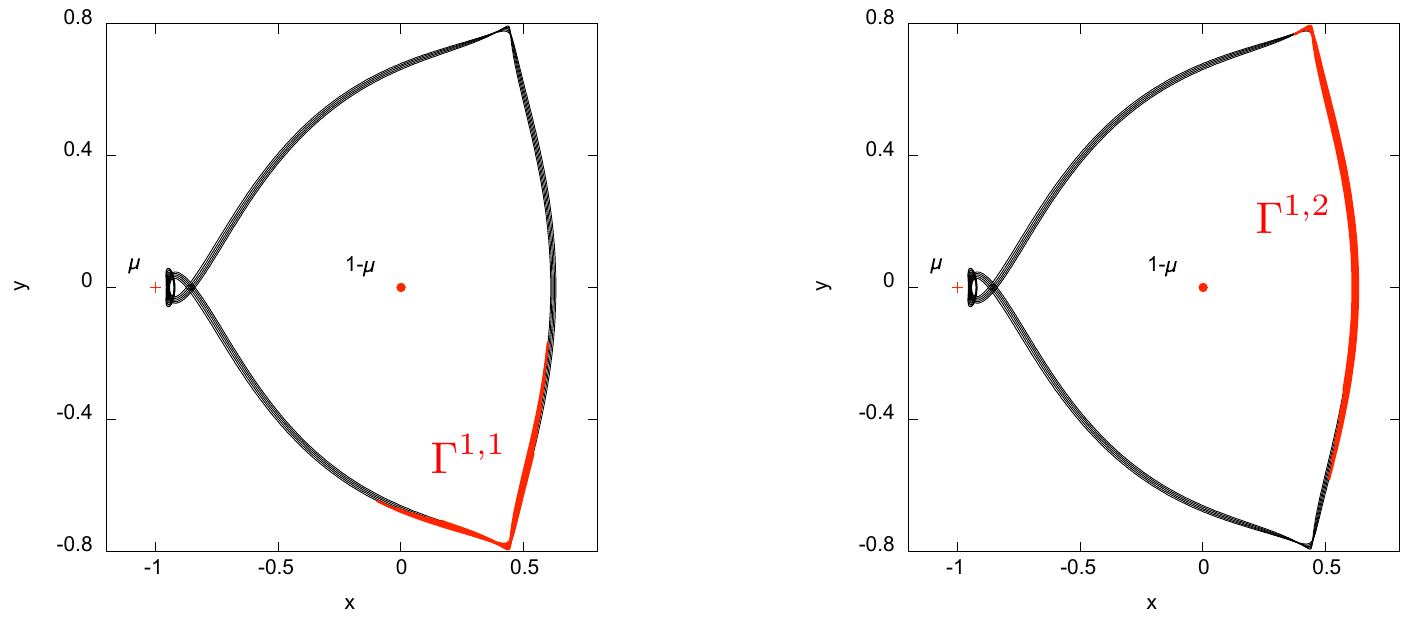} \\
\includegraphics[height=2.0in]{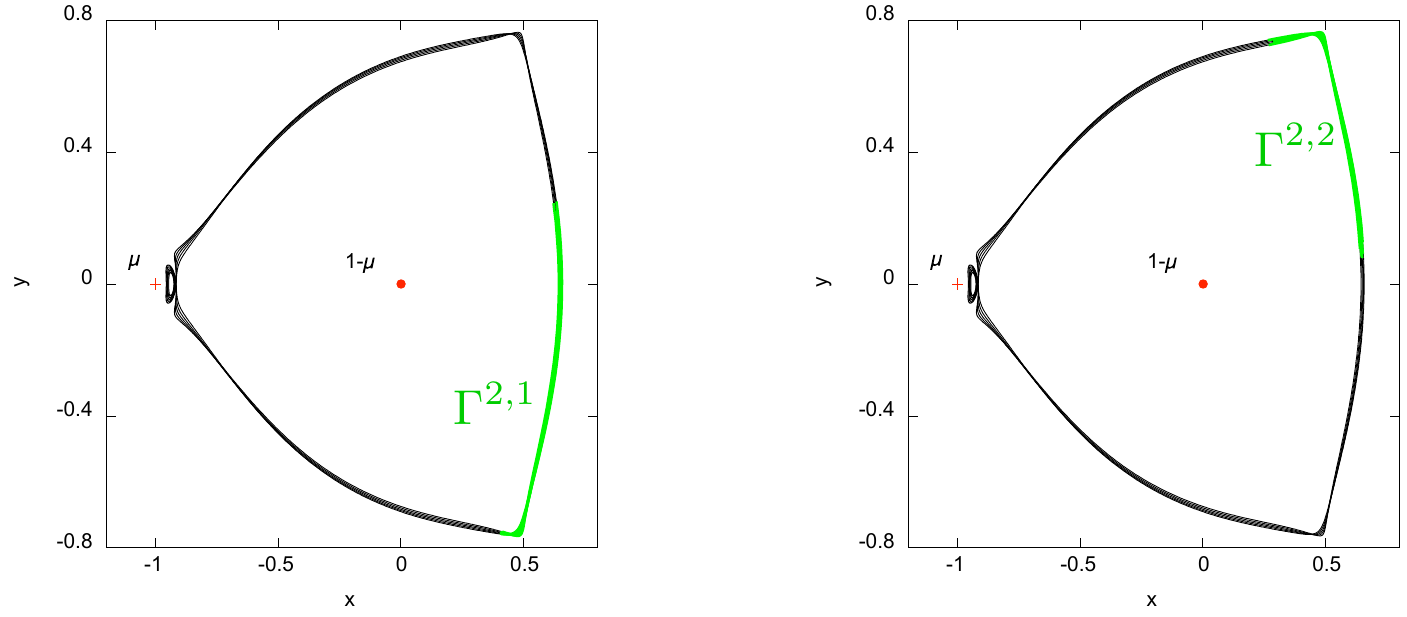}%
\end{tabular}%
\end{center}
\caption{Homoclinic orbit $\Phi _{t}(p_{1}(x^{\ast }))$ (top) and $\Phi
_{t}(p_{2}(x^{\ast }))$ (bottom) in black, together with the homoclinic
channels $\Gamma ^{i,j}$, in green and red, for $x^{\ast }\in \mathbf{I}%
=\left( -0.955,-0.945\right) $.}
\label{fig:Gamma}
\end{figure}
\begin{figure}[tbp]
\begin{center}
\includegraphics[height=1.55in]{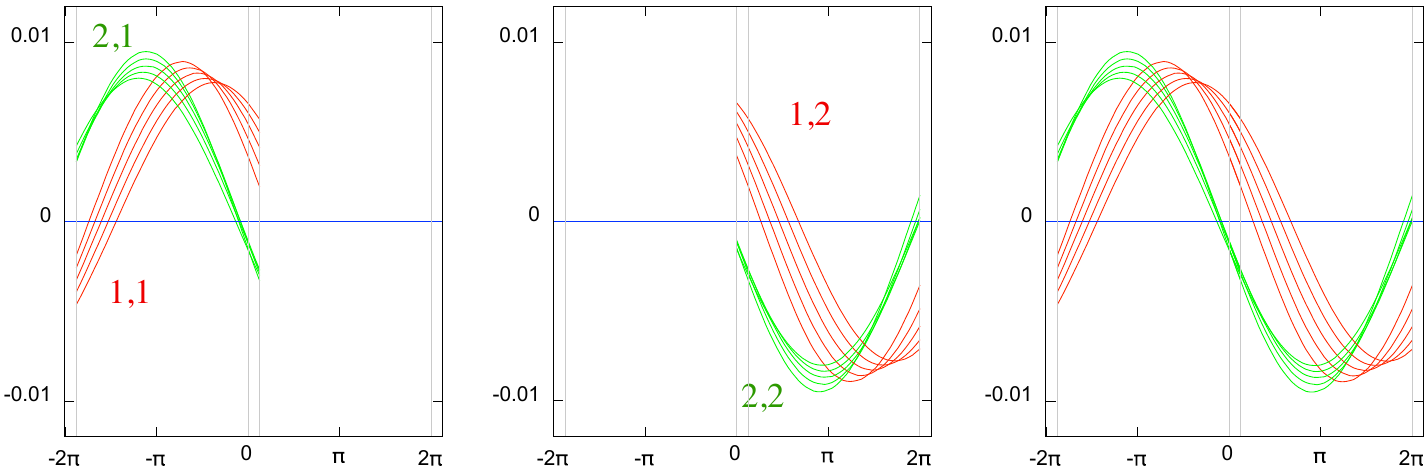}
\end{center}
\caption{The plots of $\protect\theta \rightarrow \frac{\partial S_{0,%
\protect\tau ^{\ast }}^{i,j}}{\partial \protect\theta }(s_{0}^{i,j}(x^{\ast
},\protect\theta )),$ for $\protect\tau ^{\ast }=0$ and $x^{\ast
}=-0.955,-0.9525,-0.95,-0.9475,-0.945$. The corresponding indexes $i,j$ are
indicated on the plot.}
\label{fig:meln2d}
\end{figure}

We have focused on the particular value $x^{\ast }=-0.95$ only because it is
a round number, and because for such $x^{\ast }$ the energy of the Lyapunov
orbit $L(x^{\ast })$ is close to the energy of the comet Oterma. For
different $x^{\ast }$ one observes similar plots. For instance, for any $%
x^{\ast }$ within a $5\cdot 10^{-3}$ distance of $-0.95$ one obtains plots
with the same properties (see Figures \ref{fig:Gamma} and \ref{fig:meln2d}).
This means that we can choose%
\begin{equation*}
\mathbf{I}=\left( -0.955,-0.945\right) .
\end{equation*}

In Figure \ref{fig:meln2d} we see that for any point on $L(x^{\ast })$ for $%
x^{\ast }\in \mathbf{I}$, by taking $\theta \leq \frac{\pi }{8}$ we can
always choose a scattering map for which condition (\ref{eq:scatter-grad-1})
is satisfied. Similarly, condition (\ref{eq:scatter-grad-2}) is satisfied
for $\theta \geq 0$. For any point on $L(x^{\ast })$ we can therefore choose
a scattering map satisfying (\ref{eq:scatter-grad-1}) or (\ref%
{eq:scatter-grad-2}). Thus, by Theorem \ref{th:diffusion-3bp}, the
computations give evidence of diffusion for any sufficiently small $%
\varepsilon $.

In Figure \ref{fig:Tx-Hx} we see the plots of $x^{\ast }\rightarrow
T(x^{\ast })$ and $x^{\ast }\rightarrow H(q(x^{\ast }))$. Clearly the
assumptions (\ref{eq:twist-cond}), (\ref{eq:dHq-ne-zero-1}), which are the
assumptions of Theorem \ref{th:kam-per3bp}, are reflected to hold in the
calculations.
\begin{figure}[tbp]
\begin{center}
\begin{tabular}{ll}
\includegraphics[height=2.2in]{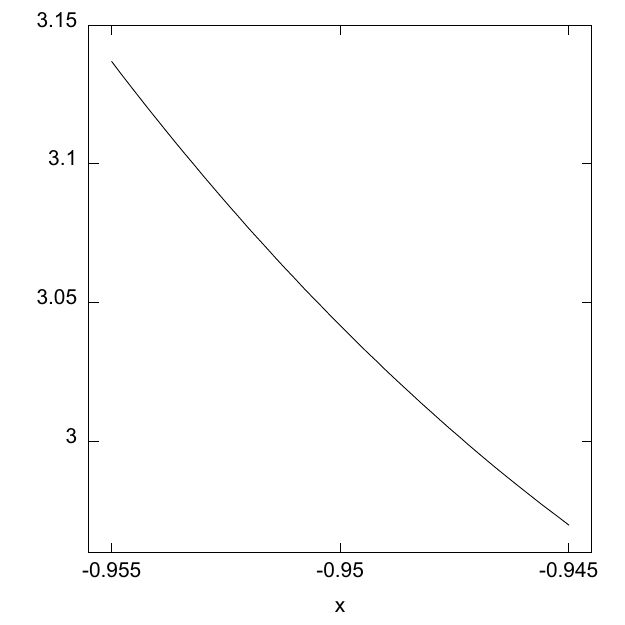} & %
\includegraphics[height=2.2in]{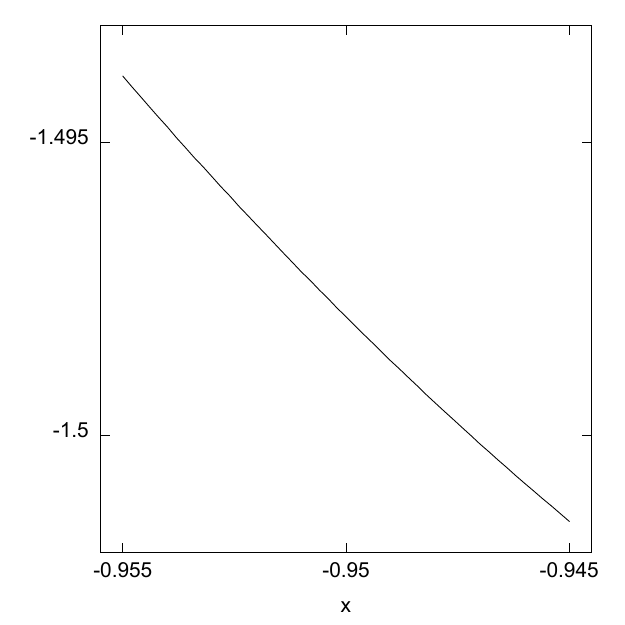}%
\end{tabular}%
\end{center}
\caption{The plots of $x^{\ast }\rightarrow T(x^{\ast })$ (left) and $%
x^{\ast }\rightarrow H(q(x^{\ast }))$ (right).}
\label{fig:Tx-Hx}
\end{figure}

The Lyapunov orbits $L(x^{\ast })$ for $x^{\ast }\in \mathbf{I}$ are
depicted in Figure \ref{fig:LapOrbits}. The width of the interval $\mathbf{I}
$, when translated to real distance in the Jupiter-Sun system, is roughly
equivalent to over $7$ million kilometers. We therefore see that our
mechanism leads to macroscopic diffusion distances. The homoclinic channels
depicted in Figure \ref{fig:Gamma} give us an idea of the paths along which
diffusion takes place and of its scale.


\section{Future work\label{sec:future-work}}

In the forthcoming work we plan to validate the assumptions of Theorem \ref{th:diffusion-3bp}. To do so the following ingredients need to be checked.

We need to check that the homoclinic channels are well defined. This means that we need to validate the conditions from Definition \ref{def:homoclinic-channel}. These will follow from the fact that the stable and unstable manifolds of the family of the Lyapunov orbits intersect transversally. Such validation has already been performed in \cite{Ca1}, hence we do not anticipate any difficulties associated with this step.

We need to verify conditions (\ref{eq:twist-cond}), (\ref{eq:dHq-ne-zero-1}), needed for the Theorem \ref{th:kam-per3bp}. Again, based on the results obtained in \cite{Ca1}, we think that this verification should be straightforward, since it involves finite computation along Lyapunov orbits.  

The last step requires the verification of conditions (\ref{eq:scatter-grad-1}), (\ref{eq:scatter-grad-2}) which are the key assumptions of Theorem \ref{th:diffusion-3bp}. This last step will likely cause most difficulties. It involves a computation of integrals over reals. Over a finite time interval these can be validated using a rigorous computer assisted enclosure of the homoclinic orbits computed in \cite{Ca1}. What will remain is to obtain an estimate on the tails of the integrals, which will need to follow from analytic arguments. We believe that such validation is possible and plan to address the problem in the future.

We believe that our approach and the mechanism presented in this paper can be used to produce a computer assisted proof of diffusion in the restricted three body problem.

\section*{Acknowledgments}
We would like to thank the two anonymous reviewers for their suggestions and comments, which helped us to improve the manuscript.

The research of M.J. Capi\'nski was partially supported by the Polish National Science Center grant 2012/05/B/ST1/00355.  
The research of M. Gidea was partially supported by the NSF grant   DMS-1515851.
The research of R. de la Llave was partially supported by the NSF grant DMS-1500943. 

\appendix

\section{Proof of Lemma \protect\ref{lem:scatter-e3bp}\label%
{sec:proof-scatter-lem}}

Let us fix $x^{\ast }\in \mathbf{I}$. In order to keep notations short, we
write $q$, $p_{1}$, $p_{2}$, $\omega _{1}$, $\omega _{2},$ $T$ for $%
q(x^{\ast })$, $p_{1}(x^{\ast })$, $p_{2}(x^{\ast })$, $\omega _{1}(x^{\ast
})$, $\omega _{2}(x^{\ast })$, $T(x^{\ast })$, respectively.

Equality (\ref{eq:s-for-3bp}) follows from (\ref{eq:s-eps-nonaut}).
Substituting (\ref{eq:wave-3bp-1}--\ref{eq:scatter-3bp}) into (\ref%
{eq:S0tau-def}) we see that%
\begin{eqnarray*}
S_{0,\tau }^{i,j}\left( x^{\ast },\theta \right) & =\int_{-\infty
}^{0}G\left( \Phi _{u}\circ \Phi _{\left( \theta +\omega _{i}\right) T/2\pi
}\left( p_{i}\right) ,\tau +u\right) \\
& \qquad \qquad \qquad -G\left( \Phi _{u}\circ k_{0}\left( x^{\ast },\theta
+2\omega _{i}\right) \mathbf{,}\tau +u\right) du \\
& \quad +\int_{0}^{+\infty }G\left( \Phi _{u}\circ \Phi _{\left( \theta
+\omega _{i}\right) T/2\pi }\left( p_{i}\right) ,\tau +u\right) \\
& \qquad \qquad \qquad -G\left( \Phi _{u}\circ k_{0}\left( x^{\ast },\theta
\right) \mathbf{,}\tau +u\right) du \\
& =\int_{-\infty }^{\theta T/2\pi }G\left( \Phi _{u+\omega _{i}T/2\pi
}\left( p_{i}\right) ,\tau +u-\theta T/2\pi \right) \\
& \qquad \qquad \qquad -G\left( \Phi _{u+2\omega _{i}T/2\pi }(q)\mathbf{,}%
\tau +u-\theta T/2\pi \right) du \\
& \quad +\int_{\theta T/2\pi }^{+\infty }G\left( \Phi _{u+\omega _{i}T/2\pi
}\left( p_{i}\right) ,\tau +u-\theta T/2\pi \right) \\
& \qquad \qquad \qquad -G\left( \Phi _{u}\left( q\right) ,\tau +u-\theta
T/2\pi \right) du.
\end{eqnarray*}%
This gives%
\begin{eqnarray*}
\frac{\partial S_{0,\tau }^{i,j}}{\partial \theta }\left( x^{\ast },\theta
\right) &=&\frac{T}{2\pi }[-G\left( \Phi _{\theta +2\omega _{i}T/2\pi
}\left( q\right) ,\tau \right) +G\left( \Phi _{\theta }\left( q\right) ,\tau
\right) \\
&&-\int_{-\infty }^{\theta T/2\pi }\frac{\partial G}{\partial t}\left( \Phi
_{u+\omega _{i}T/2\pi }\left( p_{i}\right) ,\tau +u-\theta T/2\pi \right) \\
&&\qquad \qquad -\frac{\partial G}{\partial t}\left( \Phi _{u+2\omega
_{i}}\left( q\right) ,\tau +u-\theta T/2\pi \right) du \\
&&-\int_{\theta T/2\pi }^{\infty }\frac{\partial G}{\partial t}\left( \Phi
_{u+\omega _{i}T/2\pi }\left( p_{i}\right) ,\tau +u-\theta T/2\pi \right) \\
&&\qquad \qquad -\frac{\partial G}{\partial t}\left( \Phi _{u}\left(
q\right) ,\tau +u-\theta T/2\pi \right) du],
\end{eqnarray*}%
hence by (\ref{eq:s0-c3bp}),%
\begin{eqnarray*}
& \frac{\partial S_{0,\tau }^{i,j}}{\partial \theta }\left(
s_{0}^{i,j}\left( x^{\ast },\theta \right) \right) \\
& =\frac{T}{2\pi }[-G\left( \Phi _{\theta }\left( q\right) ,\tau \right)
+G\left( \Phi _{\theta -2\omega _{i}T/2\pi }\left( q\right) ,\tau \right) \\
& \quad -\int_{-\infty }^{\left( \theta -2\omega _{i}\right) T/2\pi }\frac{%
\partial G}{\partial t}\left( \Phi _{u+\omega _{i}T/2\pi }\left(
p_{i}\right) ,\tau +u-\left( \theta -2\omega _{i}\right) T/2\pi \right) \\
& \qquad \qquad -\frac{\partial G}{\partial t}\left( \Phi _{u+2\omega
_{i}T/2\pi }\left( q\right) ,\tau +u-\left( \theta -2\omega _{i}\right)
T/2\pi \right) du \\
& \quad -\int_{\left( \theta -2\omega _{i}\right) T/2\pi }^{\infty }\frac{%
\partial G}{\partial t}\left( \Phi _{u+\omega _{i}T/2\pi }\left(
p_{i}\right) ,\tau +u-\left( \theta -2\omega _{i}\right) T/2\pi \right) \\
& \qquad \qquad -\frac{\partial G}{\partial t}\left( \Phi _{u}\left(
q\right) ,\tau +u-\left( \theta -2\omega _{i}\right) T/2\pi \right) du] \\
& =\frac{T}{2\pi }[-G\left( \Phi _{\theta }\left( q\right) ,\tau \right)
+G\left( \Phi _{\theta -2\omega _{i}T/2\pi }\left( q\right) ,\tau \right) \\
& \quad -\int_{-\infty }^{0}\frac{\partial G}{\partial t}\left( \Phi
_{u+\left( \theta -\omega _{i}\right) T/2\pi }\left( p_{i}\right) ,\tau
+u\right) -\frac{\partial G}{\partial t}\left( \Phi _{u+\theta T/2\pi
}\left( q\right) ,\tau +u\right) du \\
& \quad -\int_{0}^{\infty }\frac{\partial G}{\partial t}\left( \Phi
_{u+\left( \theta -\omega _{i}\right) T/2\pi }\left( p_{i}\right) ,\tau
+u\right) \\
& \quad -\frac{\partial G}{\partial t}\left( \Phi _{u+\left( \theta
-2\omega _{i}\right) T/2\pi }\left( q\right) ,\tau +u\right) du],
\end{eqnarray*}%
which concludes our proof. \hfill $\blacksquare $

\section*{References}
\bibliography{bib}{}

\begin{thebibliography}{10}

\bibitem{AM}
Ralph Abraham and Jerrold~E. Marsden.
\newblock {\em Foundations of mechanics}.
\newblock Benjamin/Cummings Publishing Co., Inc., Advanced Book Program,
  Reading, Mass., 1978.
\newblock Second edition, revised and enlarged, With the assistance of Tudor
  Ra{\c{t}}iu and Richard Cushman.

\bibitem{Bates99}
Peter~W. Bates, Kening Lu, and Chongchun Zeng.
\newblock Persistence of overflowing manifolds for semiflow.
\newblock {\em Comm. Pure Appl. Math.}, 52(8):983--1046, 1999.

\bibitem{Berger13}
Pierre Berger and Abed Bounemoura.
\newblock A geometrical proof of the persistence of normally hyperbolic
  submanifolds.
\newblock {\em Dyn. Syst.}, 28(4):567--581, 2013.

\bibitem{Broucke}
Roger Broucke.
\newblock Periodic orbits in the restricted three-body problem with earth-moon
  masses.
\newblock {\em Jet Propulsion Laboratory, California Institute of Technology},
  1968.

\bibitem{Canalias}
Elisabet Canalias, Amadeu Delshams, Joseph.~J. Masdemont, and Pablo Rold{\'a}n.
\newblock The scattering map in the planar restricted three body problem.
\newblock {\em Celestial Mech. Dynam. Astronom.}, 95(1-4):155--171, 2006.

\bibitem{Capinski2009}
Maciej~J. Capi{\'n}ski.
\newblock Covering relations and the existence of topologically normally
  hyperbolic invariant sets.
\newblock {\em Discrete and Continuous Dynamical Systems}, 23(3):705--725,
  2009.

\bibitem{Ca1}
Maciej~J. Capi{\'n}ski.
\newblock Computer assisted existence proofs of {L}yapunov orbits at {$L\sb 2$}
  and transversal intersections of invariant manifolds in the {J}upiter-{S}un
  {PCR}3{BP}.
\newblock {\em SIAM J. Appl. Dyn. Syst.}, 11(4):1723--1753, 2012.

\bibitem{Ca3}
Maciej~J. Capi{\'n}ski and Pablo Rold{\'a}n.
\newblock Existence of a center manifold in a practical domain around {$L_1$}
  in the restricted three-body problem.
\newblock {\em SIAM J. Appl. Dyn. Syst.}, 11(1):285--318, 2012.

\bibitem{Ca2}
Maciej~J. Capi{\'n}ski and Carles Sim{\'o}.
\newblock Computer assisted proof for normally hyperbolic invariant manifolds.
\newblock {\em Nonlinearity}, 25(7):1997--2026, 2012.

\bibitem{CapinkiZ2011}
Maciej~J. Capi{\'n}ski and Piotr Zgliczy{\'n}ski.
\newblock Cone conditions and covering relations for topologically normally
  hyperbolic invariant manifolds.
\newblock {\em Discrete and Continuous Dynamical Systems}, 30(3):641--670,
  2011.

\bibitem{CZ}
Maciej~J. Capi{\'n}ski and Piotr Zgliczy{\'n}ski.
\newblock Transition tori in the planar restricted elliptic three-body problem.
\newblock {\em Nonlinearity}, 24(5):1395--1432, 2011.

\bibitem{CapinskiZ2015}
Maciej~J. Capi{\'n}ski and Piotr Zgliczy{\'n}ski.
\newblock Geometric proof for normally hyperbolic invariant manifolds.
\newblock {\em Journal of Differential Equations}, 259(11):6215 -- 6286, 2015.

\bibitem{DLS3}
Amadeu Delshams, Rafael de~la Llave, and Tere~M. Seara.
\newblock A geometric approach to the existence of orbits with unbounded energy
  in generic periodic perturbations by a potential of generic geodesic flows of
  {${\bf T}\sp 2$}.
\newblock {\em Comm. Math. Phys.}, 209(2):353--392, 2000.

\bibitem{DLS2}
Amadeu Delshams, Rafael de~la Llave, and Tere~M. Seara.
\newblock A geometric mechanism for diffusion in {H}amiltonian systems
  overcoming the large gap problem: heuristics and rigorous verification on a
  model.
\newblock {\em Mem. Amer. Math. Soc.}, 179(844):viii+141, 2006.

\bibitem{DLS1}
Amadeu Delshams, Rafael de~la Llave, and Tere~M. Seara.
\newblock Geometric properties of the scattering map of a normally hyperbolic
  invariant manifold.
\newblock {\em Adv. Math.}, 217(3):1096--1153, 2008.

\bibitem{DelshamsGLS2008}
Amadeu Delshams, Marian Gidea, Rafael de~la Llave, and Tere~M. Seara.
\newblock Geometric approaches to the problem of instability in {H}amiltonian
  systems. {A}n informal presentation.
\newblock In {\em Hamiltonian dynamical systems and applications}, NATO Sci.
  Peace Secur. Ser. B Phys. Biophys., pages 285--336. Springer, Dordrecht,
  2008.

\bibitem{DGR10}
Amadeu Delshams, Marian Gidea, and Pablo Rold{\'a}n.
\newblock Transition map and shadowing lemma for normally hyperbolic invariant
  manifolds.
\newblock {\em Discrete Contin. Dyn. Syst.}, 33(3):1089--1112, 2013.

\bibitem{DGR15}
Amadeu Delshams, Marian Gidea, and Pablo Roldan.
\newblock Arnold's mechanism of diffusion in the spatial circular restricted
  three-body problem: {A} semi-analytical argument.
\newblock {\em Phys. D}, 334:29--48, 2016.

\bibitem{delaRosa}
Amadeu Delshams, Vadim Kaloshin, Abraham de~la Rosa, and Tere~M. Seara.
\newblock Global instability in the elliptic restricted three body problem.
\newblock 2015.

\bibitem{Eldering12}
Jaap Eldering.
\newblock Persistence of noncompact normally hyperbolic invariant manifolds in
  bounded geometry.
\newblock {\em C. R. Math. Acad. Sci. Paris}, 350(11-12):617--620, 2012.

\bibitem{Kaloshin}
Jacques Fejoz, Marcel Guardia, Vadim Kaloshin, and Pablo Roldan.
\newblock Kirkwood gaps and diffusion along mean motion resonances in the
  restricted planar three-body problem.
\newblock {\em http://arxiv.org/abs/1109.2892}, 09 2011.

\bibitem{Fenichel}
Neil Fenichel.
\newblock Persistence and smoothness of invariant manifolds for flows.
\newblock {\em Indiana Univ. Math. J.}, 21:193--226, 1971/1972.

\bibitem{Fenichel2}
Neil Fenichel.
\newblock Asymptotic stability with rate conditions.
\newblock {\em Indiana Univ. Math. J.}, 23:1109--1137, 1973/74.

\bibitem{Fenichel4}
Neil Fenichel.
\newblock Asymptotic stability with rate conditions. {II}.
\newblock {\em Indiana Univ. Math. J.}, 26(1):81--93, 1977.

\bibitem{Fig}
Jordi-Llu{\'{\i}}s Figueras and {\`A}lex Haro.
\newblock Reliable computation of robust response tori on the verge of
  breakdown.
\newblock {\em SIAM J. Appl. Dyn. Syst.}, 11(2):597--628, 2012.

\bibitem{GalanteK1}
Joseph Galante and Vadim Kaloshin.
\newblock Destruction of invariant curves in the restricted circular planar
  three-body problem by using comparison of action.
\newblock {\em Duke Math. J.}, 159(2):275--327, 2011.

\bibitem{GT}
V~Gelfreich and D~Turaev.
\newblock Arnold diffusion in a priory chaotic hamiltonian systems.
\newblock {\em http://arxiv.org/abs/1406.2945v2}, 2014.

\bibitem{GideaLlaveSeara14}
Marian Gidea, Rafael de~la Llave, and Tere Seara.
\newblock A general mechanism of diffusion in hamiltonian systems: Qualitative
  results.
\newblock {\em http://arxiv.org/abs/1405.0866}, 2014.

\bibitem{HPS}
M.~W. Hirsch, C.~C. Pugh, and M.~Shub.
\newblock {\em Invariant manifolds}.
\newblock Lecture Notes in Mathematics, Vol. 583. Springer-Verlag, Berlin-New
  York, 1977.

\bibitem{Laplace1}
Marquis~de La~Place.
\newblock {\em Celestial mechanics. {V}ols. {I}--{IV}}.
\newblock Translated from the French, with a commentary, by Nathaniel Bowditch.
  Chelsea Publishing Co., Inc., Bronx, N.Y., 1966.

\bibitem{Laplace2}
Marquis~de La~Place.
\newblock {\em Celestial mechanics. {V}ol. {V}}.
\newblock In French. Chelsea Publishing Co., Inc., Bronx, N.Y., 1969.

\bibitem{Sun}
Jie Liu and Yi~Sui Sun.
\newblock Chaotic motion of comets in near-parabolic orbit: mapping approaches.
\newblock {\em Celestial Mech. Dynam. Astronom.}, 60(1):3--28, 1994.

\bibitem{Simo}
Jaume Llibre, Regina Mart{\'{\i}}nez, and Carles Sim{\'o}.
\newblock Tranversality of the invariant manifolds associated to the {L}yapunov
  family of periodic orbits near {$L\sb 2$} in the restricted three-body
  problem.
\newblock {\em J. Differential Equations}, 58(1):104--156, 1985.

\bibitem{Marchal}
Christian Marchal.
\newblock {\em The three-body problem}, volume~4 of {\em Studies in
  Astronautics}.
\newblock Elsevier Science Publishers, B.V., Amsterdam, 1990.
\newblock With a foreword by Victor Szebehely, With French, Russian, German,
  Spanish, Japanese, Chinese and Arabic summaries.

\bibitem{Marco16}
J.-P. {Marco}.
\newblock {Arnold diffusion for cusp-generic nearly integrable convex systems
  on ${\mathbb A}^3$}.
\newblock {\em ArXiv e-prints}, February 2016.

\bibitem{Moser}
J{\"u}rgen Moser.
\newblock On the generalization of a theorem of {A}. {L}iapounoff.
\newblock {\em Comm. Pure Appl. Math.}, 11:257--271, 1958.

\bibitem{Principa}
Isaac Newton.
\newblock {\em The {\it {P}rincipia}: mathematical principles of natural
  philosophy}.
\newblock University of California Press, Berkeley, CA, 1999.
\newblock A new translation by I. Bernard Cohen and Anne Whitman, assisted by
  Julia Budenz, Preceded by ``A guide to Newton's {{\i}t Principia}'' by Cohen.

\bibitem{Sussmann}
H{\'e}ctor~J. Sussmann.
\newblock A general theorem on local controllability.
\newblock {\em SIAM J. Control Optim.}, 25(1):158--194, 1987.

\bibitem{Szebehely}
Victor~G. Szebehely.
\newblock {\em Theory of Orbits in the Restricted Problem of Three Bodies}.
\newblock Academic Press, 1967.

\bibitem{GalanteU}
John~C. Urschel and Joseph~R. Galante.
\newblock Instabilities in the {S}un-{J}upiter-asteroid three body problem.
\newblock {\em Celestial Mech. Dynam. Astronom.}, 115(3):233--259, 2013.

\bibitem{Jay}
Jan~Bouwe van~den Berg, Andr{\'e}a Desch{\^e}nes, Jean-Philippe Lessard, and
  Jason~D. Mireles~James.
\newblock Stationary coexistence of hexagons and rolls via rigorous
  computations.
\newblock {\em SIAM J. Appl. Dyn. Syst.}, 14(2):942--979, 2015.

\bibitem{nhims}
Stephen Wiggins.
\newblock {\em Normally hyperbolic invariant manifolds in dynamical systems},
  volume 105 of {\em Applied Mathematical Sciences}.
\newblock Springer-Verlag, New York, 1994.
\newblock With the assistance of Gy{\"o}rgy Haller and Igor Mezi{\'c}.

\bibitem{W3}
Daniel Wilczak.
\newblock The existence of {S}hilnikov homoclinic orbits in the {M}ichelson
  system: a computer assisted proof.
\newblock {\em Found. Comput. Math.}, 6(4):495--535, 2006.

\bibitem{W2}
Daniel Wilczak.
\newblock Abundance of heteroclinic and homoclinic orbits for the hyperchaotic
  {R}\"ossler system.
\newblock {\em Discrete Contin. Dyn. Syst. Ser. B}, 11(4):1039--1055, 2009.

\bibitem{WZ}
Daniel Wilczak and Piotr Zgliczynski.
\newblock Heteroclinic connections between periodic orbits in planar restricted
  circular three-body problem---a computer assisted proof.
\newblock {\em Comm. Math. Phys.}, 234(1):37--75, 2003.

\bibitem{W1}
Daniel Wilczak and Piotr Zgliczy{\'n}ski.
\newblock Computer assisted proof of the existence of homoclinic tangency for
  the {H}\'enon map and for the forced damped pendulum.
\newblock {\em SIAM J. Appl. Dyn. Syst.}, 8(4):1632--1663, 2009.

\bibitem{Xue}
Jinxin Xue.
\newblock Arnold diffusion in a restricted planar four-body problem.
\newblock {\em Nonlinearity}, 27(12):2887, 2014.

\end{thebibliography}
\bibliographystyle{plain}

\end{document}